\documentclass[11pt]{article}
\usepackage{amsmath}
\usepackage{amsfonts}
\usepackage{chngpage}
\usepackage{titlesec}
\usepackage{fancyhdr}
\usepackage[dvips]{color}
\usepackage{remreset}
\usepackage{epsfig}
\usepackage{multirow}

\numberwithin{figure}{section}
\makeatletter
\def\EquationsBySection{\def\theequation
{\thesection.\arabic{equation}}%
\@addtoreset{equation}{section}}

\makeatother \EquationsBySection
\allowdisplaybreaks

\newtheorem{theorem}{\bf Theorem}[section]
\newtheorem{proposition}[theorem]{\bf Proposition}

\newtheorem{lemma}[theorem]{\bf Lemma}

\newenvironment{proof}{Proof:}{\quad \hfill $\Box$ \vspace{2ex}}

\newtheorem{example}[theorem]{\bf Example}

\begin{document}

\renewcommand\thepage{{\arabic{page}}}

{\renewcommand\thesection
    {{\thesection}}}

\changetext{}{}{-1.3cm}{-1.3cm}{}

\newcommand{\rme}{\mathrm{e}}
\newcommand{\rmi}{\mathrm{i}}
\newcommand{\rmd}{\mathrm{d}}
\newcommand{\norm}[1]{\left\Vert#1\right\Vert}
\newcommand{\abs}[1]{\left\vert#1\right\vert}
\newcommand{\set}[1]{\left\{#1\right\}}
\newcommand{\inner}[1]{\left\langle#1\right\rangle}
\newcommand{\Real}{\mathbb R}
\newcommand{\eps}{\varepsilon}
\newcommand{\To}{\rightarrow}
\newcommand{\BX}{\mathbf{B}(X)}
\newcommand{\A}{\mathcal{A}}
\renewcommand{\vec}[1]{\boldsymbol{#1}}
%


\title{Computing Integrals Involved the Gaussian Function with a Small Standard Deviation
\thanks{This research was supported in part by the Ministry of Science and Technology of China under grant 2016YFB0200602, by the Natural Science Foundation of China under grants 11471013 and 11771464, and by the US National Science Foundation under grant DMS-1522332.
}}
\author{
Yunyun Ma\thanks{School of Computer Science and Network Security,
Dongguan University of Technology, Dongguan 523808, P. R. China.
{\it mayy007@foxmail.com}. }
\ and
Yuesheng Xu\thanks{Department of Mathematics and Statistics, Old Dominion University, Norfolk, VA 23529, USA.
School of Data and Computer Science, and Guangdong
Province Key Lab of Computational Science,  Sun
Yat-sen University, Guangzhou 510275, P. R. China.
{\it y1xu@odu.edu}. All correspondence should be sent to this author.}
}
\maketitle

\begin{abstract}
We develop efficient numerical integration methods for computing an integral whose integrand is a product of a smooth function and the Gaussian function with a small standard deviation. Traditional numerical integration methods applied to the integral normally lead to poor accuracy due to the rapid change in high order derivatives of its integrand when the standard deviation is small. The proposed quadrature schemes are based on graded meshes designed according to the standard deviation so that the quadrature errors on the resulting subintervals are approximately equal. The integral in each subinterval is then computed by considering the Gaussian function as a weight function and interpolating the smooth factor of the integrand at the Chebyshev points of the first kind. For a finite order differentiable factor, we design a quadrature scheme having accuracy of a polynomial order and for an infinitely differentiable factor of the integrand, we design a quadrature scheme having accuracy of an exponential order.
Numerical results are presented to confirm the accuracy of these proposed quadrature schemes.
\end{abstract}

\medskip


\smallskip


{

\section{Introduction}

We consider in this paper developing efficient quadrature methods for an integral whose integrand is a product of a smooth function and the Gaussian function with a small standard deviation. Gaussian functions \cite{Arfken, John, Mandl, Temme,RKM} arise in many areas of mathematics, physics and engineering, especially in the fields of statistics, probability theory \cite{Feller1, Feller2}, image and signal processing \cite{Arfken, Gonzalez, Ito, Lu, Walter}. Design of numerical quadrature formulas for the integral is important for numerical analysis for the aforementioned applications.
When the standard deviation is small, high order derivatives of the Gaussian function
oscillate dramatically near the position of the peak of the curve of the Gaussian function,
and decay rapidly away from that position. Standard numerical quadratures of such an integral normally lead to poor accuracy due to the rapid change in the derivatives of its integrand when the standard deviation is small.
Computing the integral requires special effort.

Quadrature schemes that we design for computing the integral is based on graded meshes constructed according to the standard deviation so that the resulting quadrature errors on the subintervals are approximately equal. The integral in each of the subintervals is then computed by considering the Gaussian function as a weight function and interpolating the smooth factor of the integrand at the Chebyshev points of the first kind. We establish accuracy of exponential and polynomial orders for the proposed quadrature schemes according to the regularity of the smooth factor of the integrand. We present numerical results to confirm the accuracy estimates.

In the literature, quadrature formulas for the integrals involved the Gaussian function were designed according to properties of the Gaussian function.  The well-known Gauss-Hermite quadrature \cite{Abramowitz, Davis} in  numerical analysis is designed for the integrals with the Gaussian weight in an unbounded integration interval. The weights of the quadrature formulas can be computed exactly, which are integrals with the integrand defined by a product of a polynomial and the Gaussian function. We call them the moments. The computation of these moments in a bounded integration interval can be found in \cite{Steen}. The numerical quadrature for higher dimensional integration with the Gaussian weight was developed in \cite{Lu0}. However, quadrature formulas for integrals involved the Gaussian function are inadequate in the literature for the case when the standard deviation of the Gaussian function is small. Such integrals require careful numerical treatment, since higher order derivatives of the Gaussian function have oscillation with large magnitude.

The purpose of this paper is to design efficient quadrature formulas for the integrals involved the Gaussian function for the case when the standard deviation is small. Motivated by the subdivision in nonuniform fashion for the singular integrals \cite{Xu1, Xu2, Schwab} and oscillatory integrals \cite{Xu3, Xu4}, we shall divide the integration interval into subintervals according to the standard deviation of the Gaussian function, and approximate the integrals in each of the resulting subintervals by replacing the smooth factor of integrand by an interpolating polynomial. The interpolation nodes are chosen as the Chebyshev points of the first  kind \cite{Cheney,Davis,Shen,Kuan}. The weights of the quadrature formulas are computed by utilizing the error function \cite{Abramowitz, Gradshteyn, Lebedev, Temme}. The accuracy of these  quadrature formulas increases as the standard deviation of the Gaussian function tends to zero and the computational complexity is independent of  the  standard deviation of the Gaussian function.

This paper is organized in six sections. We study in Section 2 useful properties of the Gaussian function to motivate the numerical integration for the integrals involved the Gaussian function with a small standard deviation.
In Section 3, we establish a basic quadrature formula for computing the integrals involved the Gaussian function in a reference integration interval. We then propose in Section 4 two composite quadrature formulas based on dividing the integration interval into subintervals according to the standard deviation of the Gaussian function, and using the formula developed in the previous section to calculate the integrals in the subintervals.  A computation method of the weights for the quadrature formulas is provided in Section 5. We show in Section 6 numerical results to verify the numerical accuracy of the proposed quadrature formulas.

\section{Integrals Involved the Gaussian Function}

In this section, we first discuss intrinsic difficulties in numerical integration by using traditional quadrature schemes for an integral whose integrand is a product of a smooth function and the Gaussian function with a small standard deviation. This motivates
us to design graded meshes of the integration interval according to the standard deviation of the Gaussian function.

We begin with recalling the definition of the Gaussian function.
For three fixed real numbers $c_0$, $c_1$ and $c_2$, a Gaussian function is of the form
\begin{equation*}
  G(x;c_0,c_1,c_2):=c_1\exp\set{-\dfrac{(x-c_2)^2}{2c_0^2}},~\text{for}~x\in\mathbb{R},
\end{equation*}
where $c_0>0$ is  called the  standard deviation, $c_1>0$ is the maximum of $G$, and
$c_2$ is the position of the maximum of $G$ occurring.
The graph of $G$ is a symmetric bell shaped curve. The parameter $c_0$ controls the width of the bell, and $c_1$ and $c_2$ are respectively the height and the position of the peak of the bell.
The Gaussian function with $c_1:=1/(c_0\sqrt{2\pi})$  is the probability density function of a normally distributed random variable with expected value $c_2$ and variance $c_0^2$.
Without loss of generality, we assume that $c_1=1$.
The behaviours of $G$ and its derivatives are influenced by the standard deviation and the center of the peak in the graph of $G$.
We simplify the parameters of $G$ by
\begin{equation}\label{Sec2:Gauss}
 G_{\alpha,\beta}(x):={\rm e}^{-\alpha^2(x-\beta)^2}, ~\text{for}~x\in\mathbb{R},
\end{equation}
where $\alpha>0$ and $\beta$ is a real number. We let $G_\alpha:=G_{\alpha,0}$.
Clearly, $\alpha:=1/(\sqrt{2}c_0)$ is the $1/\sqrt{2}$-multiple of the reciprocal of the standard deviation of the Gaussian function.
A small standard deviation corresponds to a large parameter $\alpha$. Therefore, in the remaining part of this paper,
we shall consider the integral involved $G_{\alpha,\beta}$ for a large $\alpha$.
Specifically, we shall study numerical integration for the integral having the form
\begin{equation}\label{Sec2:Int}
  \mathcal{I}_{\alpha}[f]:=\int_If(x)G_\alpha(x){\rm d}x,
\end{equation}
where $I:=[0,1]$, $\alpha\gg1$ and $f\in C(I)$ is independent of $\alpha$.
The integrals involved $G_{\alpha,\beta}$ with $\beta\neq 0$
may be treated by first splitting the integral interval into two subintervals, on each of which $\beta$ is an end point,
and then changing the integrals defined on the subintervals into the integrals defined on $I$ having the form \eqref{Sec2:Int}. The quadrature methods to be proposed for \eqref{Sec2:Int} can extend straightforward to the general case.

\vspace{-2em}
\begin{center}
\makeatletter
\def\@captype{figure}
\makeatother
\includegraphics[width=7.0cm, height=7.8cm]{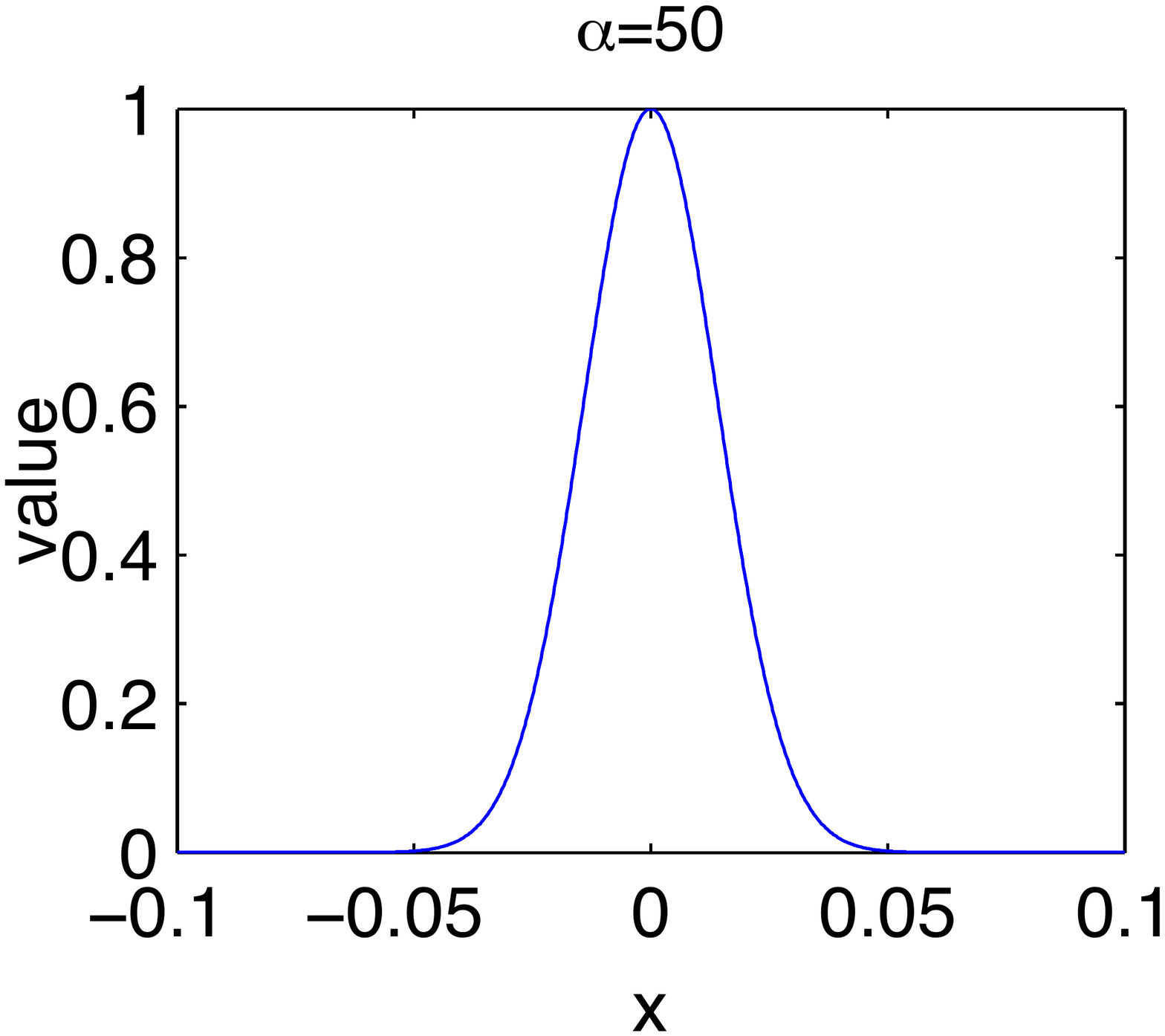}\ \hspace{2cm} \
\includegraphics[width=7.0cm, height=7.8cm]{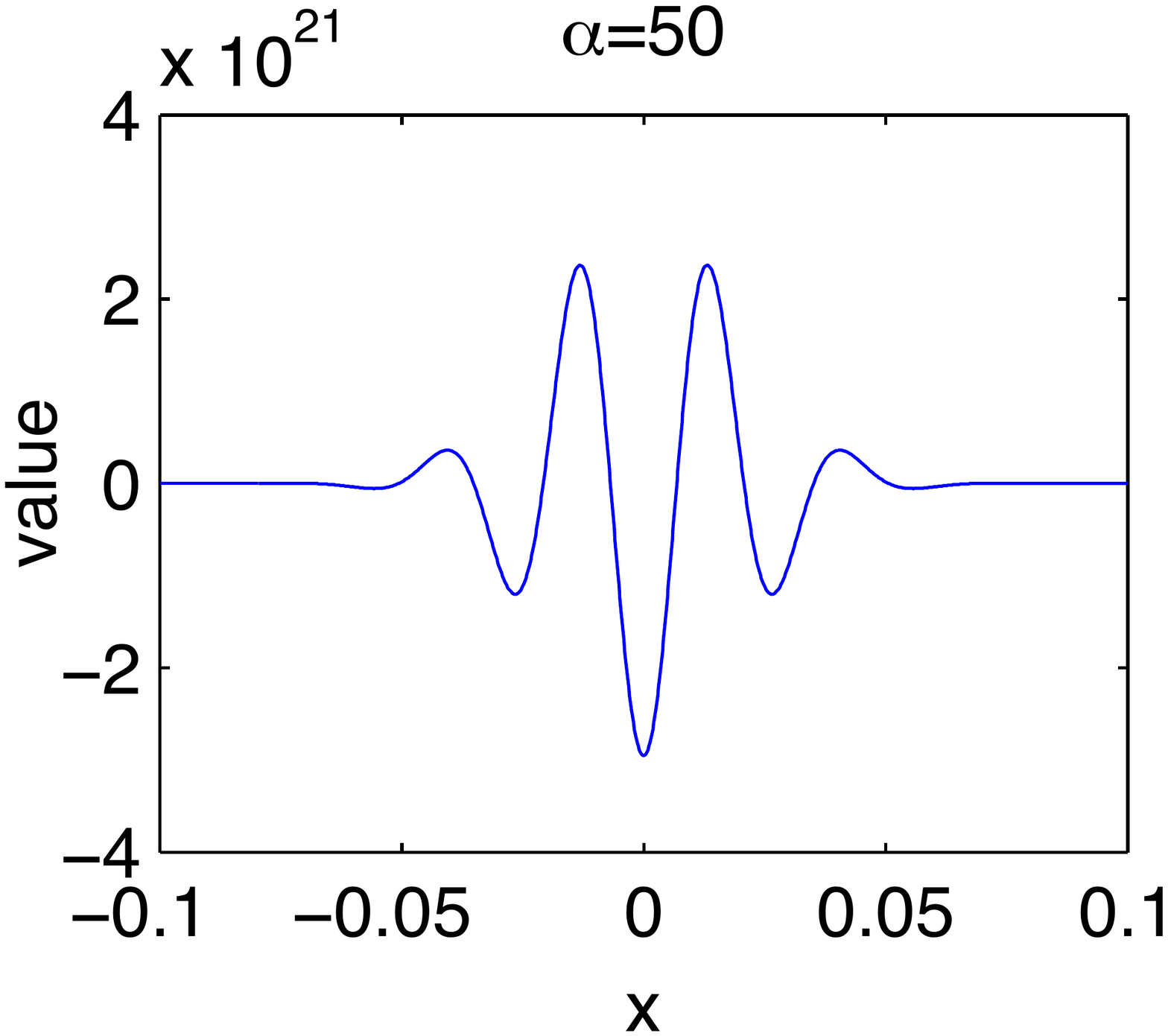}\\
\vspace{-4.5em}
\centerline{\small A: ~$m=0$ ~\hspace{7.3cm} B: ~$m=10$ }
\vspace{-1em}
{\small\caption{The derivatives of $G_\alpha$}
\label{Sec2:Fig1}}
\end{center}

\vspace{-2em}
\begin{center}
\makeatletter
\def\@captype{figure}
\makeatother
\includegraphics[width=7.0cm, height=7.8cm]{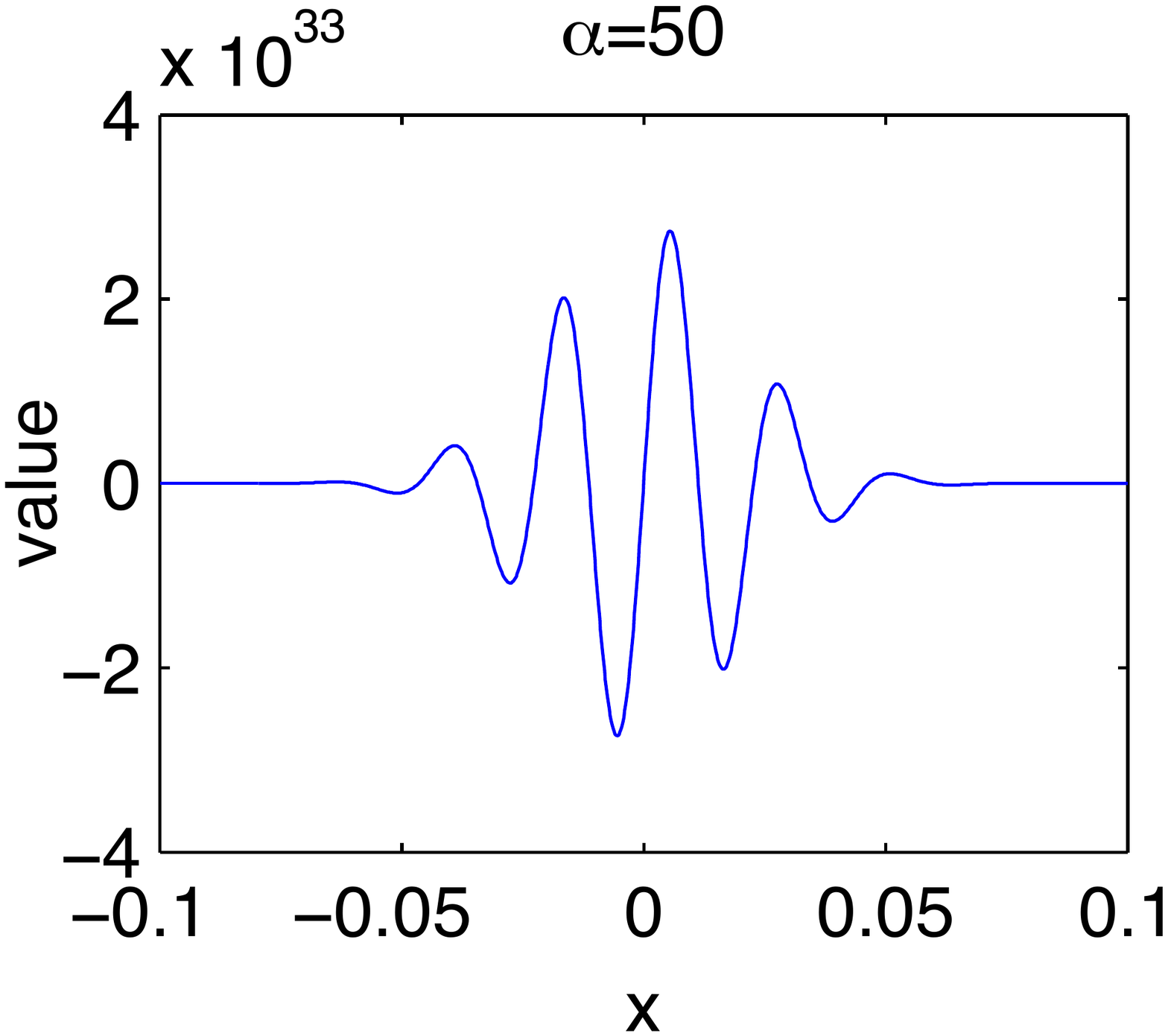}\ \hspace{2cm} \
\includegraphics[width=7.0cm, height=7.8cm]{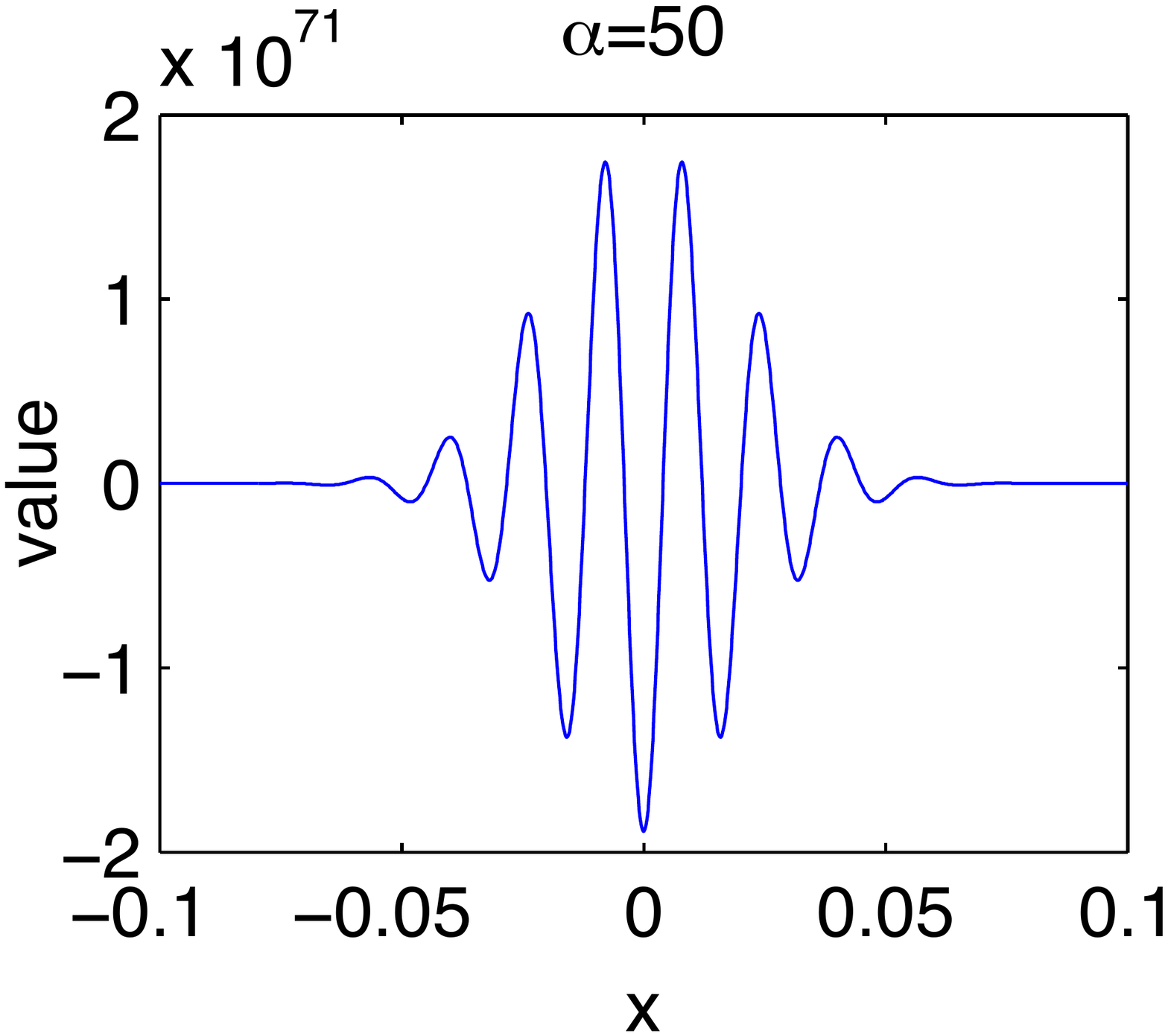}\\
\vspace{-4.5em}
\centerline{\small A: ~$m=15$ ~\hspace{7.3cm} B: ~$m=30$ }
\vspace{-1em}
{\small\caption{The derivatives of $G_\alpha$}
\label{Sec2:Fig2}}
\end{center}

We first review properties of  $G_\alpha$ and its derivatives for $\alpha>1$. Let $\mathbb{N}:=\set{1, 2,\ldots}$ and $\mathbb{N}_0:= \{0\}\cup\mathbb{N}$.
For $n\in\mathbb{N}_0$, we denote by $H_n$ the Hermite polynomials \cite{Abramowitz}.
For $m\in\mathbb{N}_0$, the $m$-th derivative of $G_\alpha$ is given by
\begin{equation}\label{Sec2:DerG}
  G^{(m)}_\alpha(x)=(-1)^m\alpha^mH_m(\alpha x)G_\alpha(x), ~\text{for}~x\in\mathbb{R}.
\end{equation}
We plot in Figures \ref{Sec2:Fig1}-\ref{Sec2:Fig3}
the graphs of $G^{(m)}_\alpha$ with $\alpha=50$ for different $m$ in the interval  $[-0.1,0.1]$, with  $m=0$, $10$, $15$, $30$, $35$ and $50$. From the graph of $G_{\alpha}$ shown in Figure \ref{Sec2:Fig1} (A), we see that it looks like a bell, with the center of the bell being zero  and the height of  the bell being one.
However,  in Figures \ref{Sec2:Fig1} (B), \ref{Sec2:Fig2}-\ref{Sec2:Fig3}, the graph of
the derivatives of $G_{\alpha}$ oscillates rapidly near zero when $m$ becomes large and the maximum of the derivatives increases fast as $m$ gets large. For the special case that $m\in\mathbb{N}_0$  is even,
by introducing the Hermite number  \cite{Kim} given by
\begin{equation*}
H_m(0)= (-1)^{m/2}2^{m/2}(m-1)!!,
\end{equation*}
we observe that the maximum of the derivatives of $G_\alpha$  grows in  an exponential order of $\alpha$, since $G^{(m)}_\alpha(0)=\alpha^mH_m(0)$. Meanwhile, the tail of the graph of the derivatives of $G_{\alpha}$ falls off quickly on the both sides and approaches the $x$-axis. The high order derivatives of the Gaussian function  behave like oscillatory functions near the  position of the peak of  its curve when $\alpha$ is large.

\begin{center}
\makeatletter
\def\@captype{figure}
\makeatother
\includegraphics[width=7.0cm, height=7.8cm]{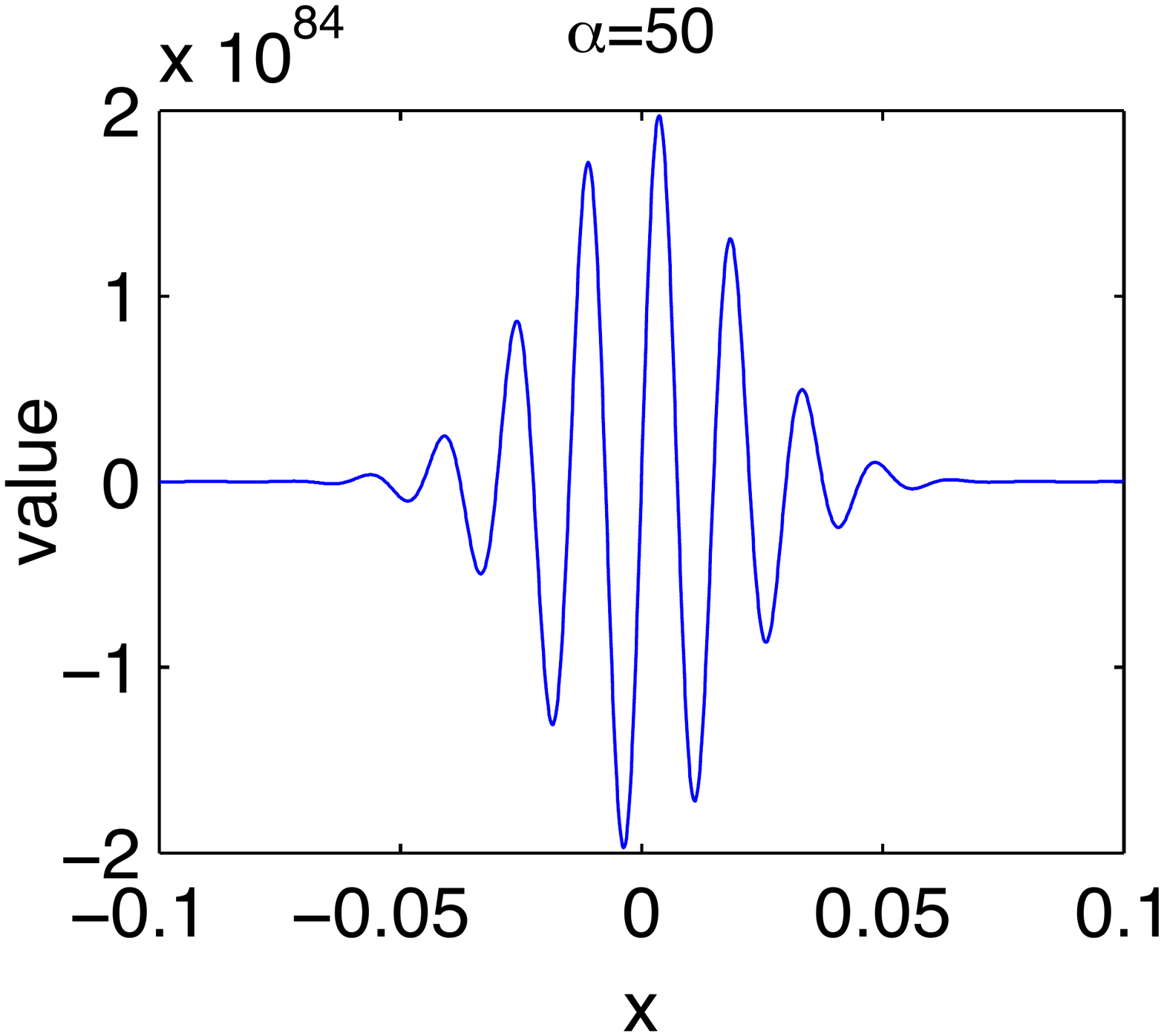}\ \hspace{2cm} \
\includegraphics[width=7.0cm, height=7.8cm]{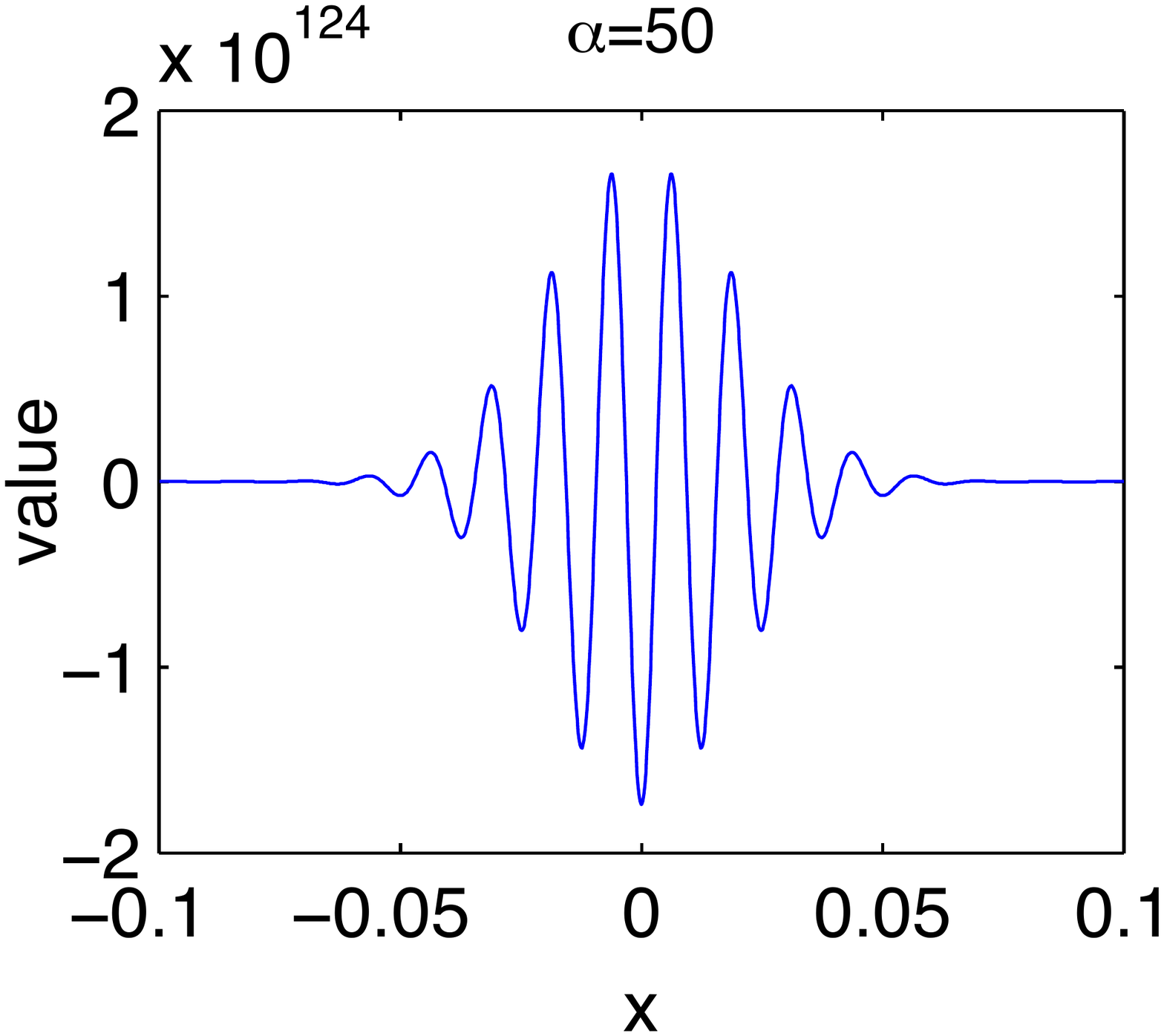}\\
\vspace{-4.5em}
\centerline{\small A: ~$m=35$ ~\hspace{7.3cm} B: ~$m=50$ }
\vspace{-1em}
{\small\caption{The derivatives of $G_\alpha$}
\label{Sec2:Fig3}}
\end{center}

\vspace{-2em}
\begin{center}
\makeatletter
\def\@captype{figure}
\makeatother
\includegraphics[width=7.0cm, height=8.5cm]{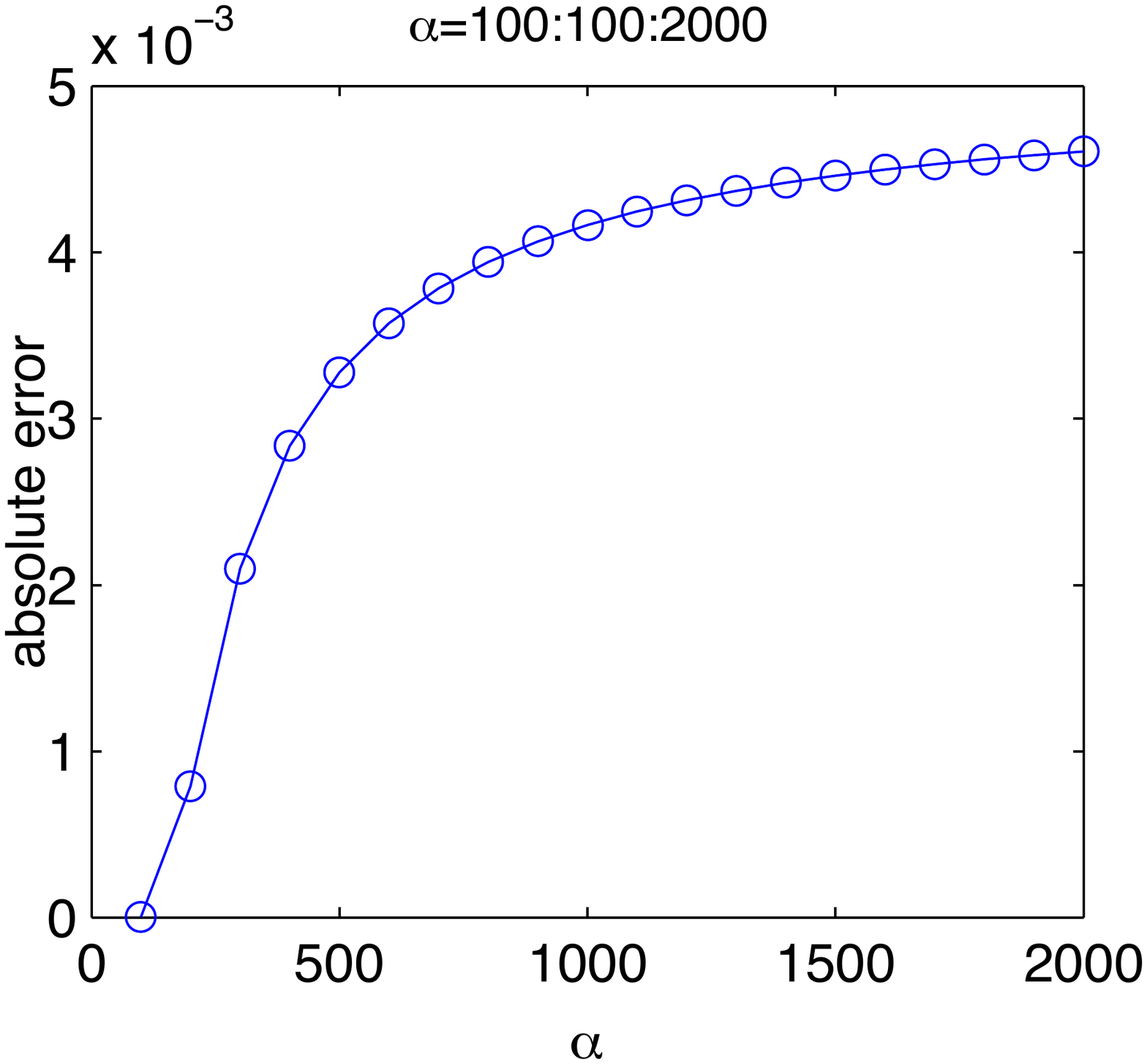}\ \hspace{1cm} \
\includegraphics[width=7.0cm, height=8.5cm]{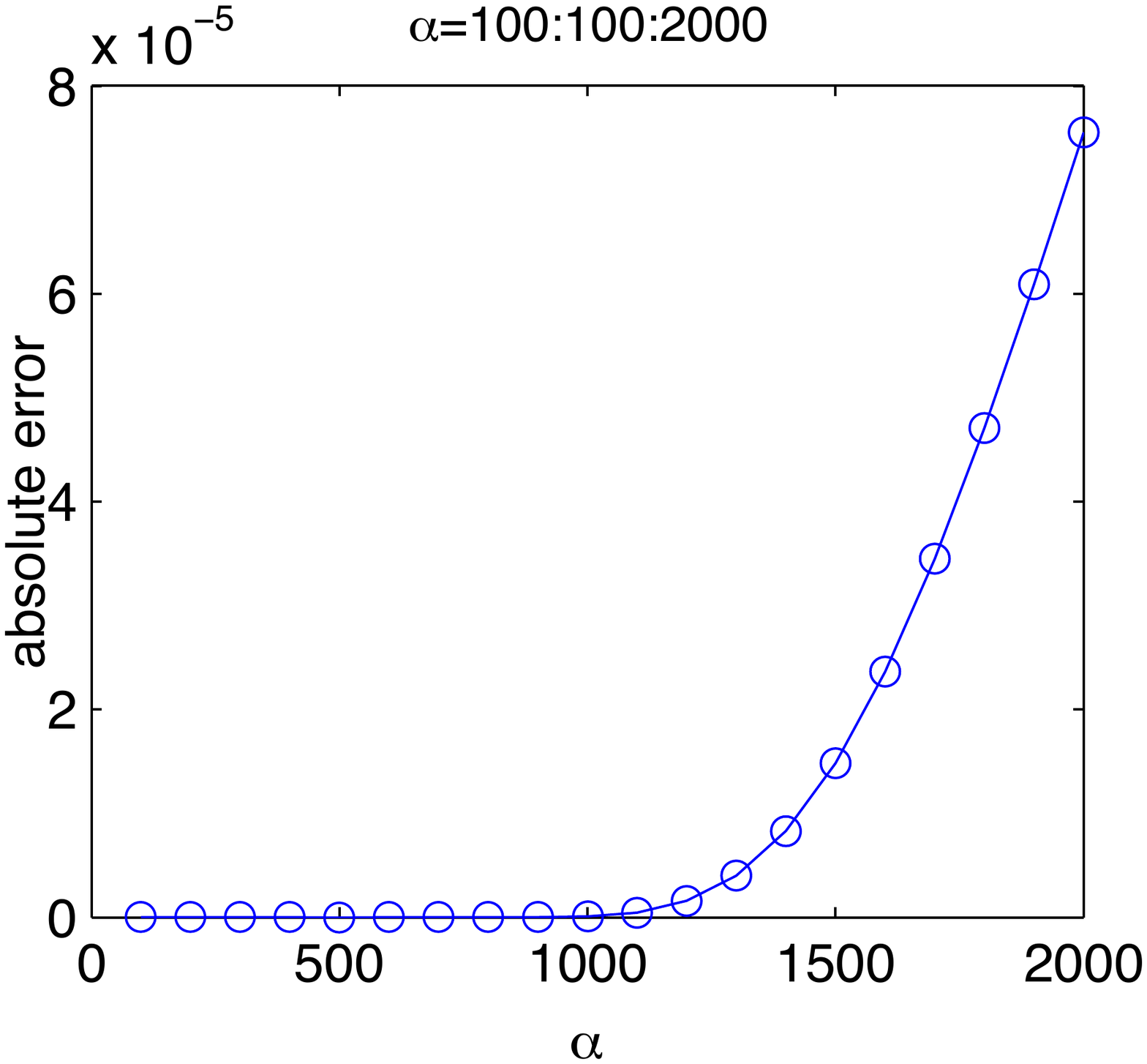}\\
\vspace{-4.5em}
\centerline{\small\ \ \ A: ~$n=100$ ~\hspace{6.5cm} B: ~$n=1000$}
\vspace{-1em}
{\small\caption{The AE of the approximation to the value of the integrals of $G_\alpha$ with fixed $n$}
\label{Sec2:Fig4}}
\end{center}

We now discuss difficulties in computing the integral \eqref{Sec2:Int} by using a traditional quadrature scheme.
As an example, we calculate the integral $\mathcal{I}_\alpha[f]$ with $f(x):=1$, $x\in I$, for different choices of $\alpha$, by employing the composite trapezoidal rule with $n$ equidistant nodes. One might expect that the trapezoidal rule would produce satisfactory numerical results since the Gaussian function is infinitely differentiable. However, we obtain rather disappointing numerical results. We plot in Figure \ref{Sec2:Fig4} the absolute error (AE) of the computed integral value
for (A) $n=100$ and (B) $n=1000$. From Figure \ref{Sec2:Fig4}, we see that for a fixed number of quadrature nodes, the accuracy of the quadrature formula decreases as $\alpha$ increases. We then plot in Figure \ref{Sec2:Fig5} the absolute error of  the computed integral value with fixed $\alpha$ for (A) $\alpha=100$, (B) $\alpha=1000$ and (C) $\alpha=10000$. We observe that in order to have certain order of accuracy, the number of quadrature nodes should increase proportional to the value of $\alpha$.

\vspace{-2.5em}
\begin{center}
\makeatletter
\def\@captype{figure}
\makeatother
\includegraphics[width=5.4cm, height=6.5cm]{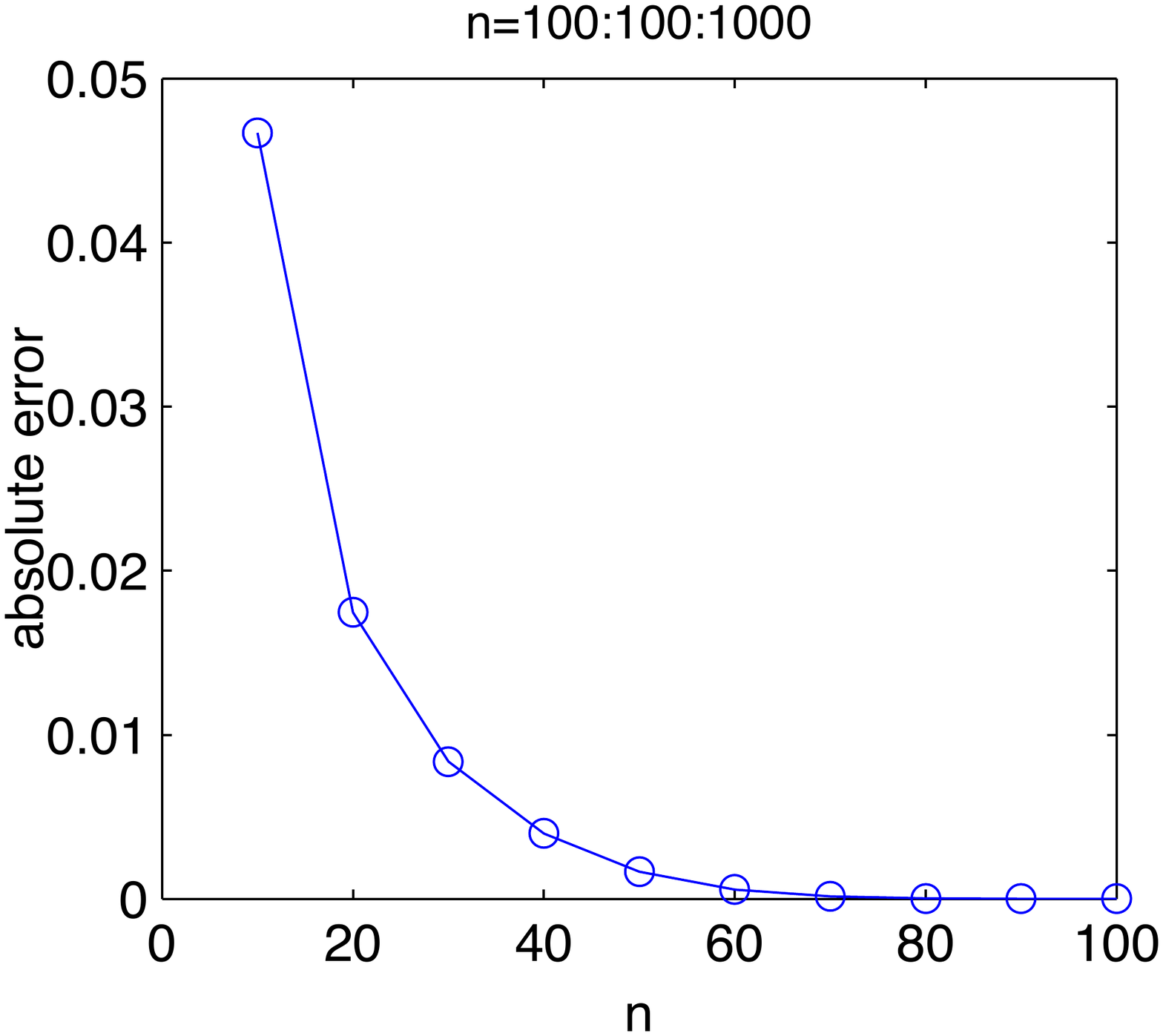}\hspace{0.6cm}
\includegraphics[width=5.4cm, height=6.5cm]{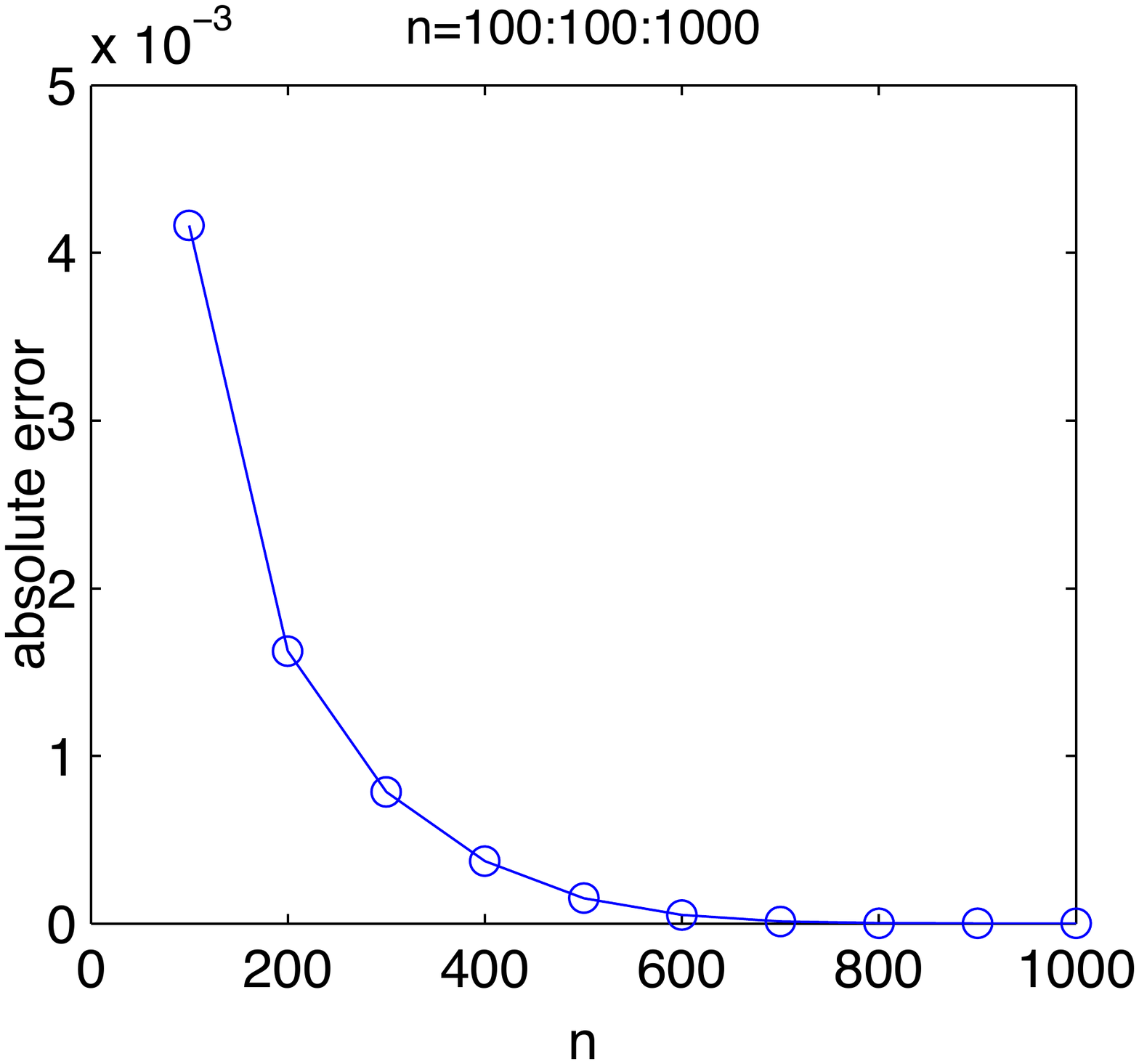}\hspace{0.6cm}
\includegraphics[width=5.4cm, height=6.5cm]{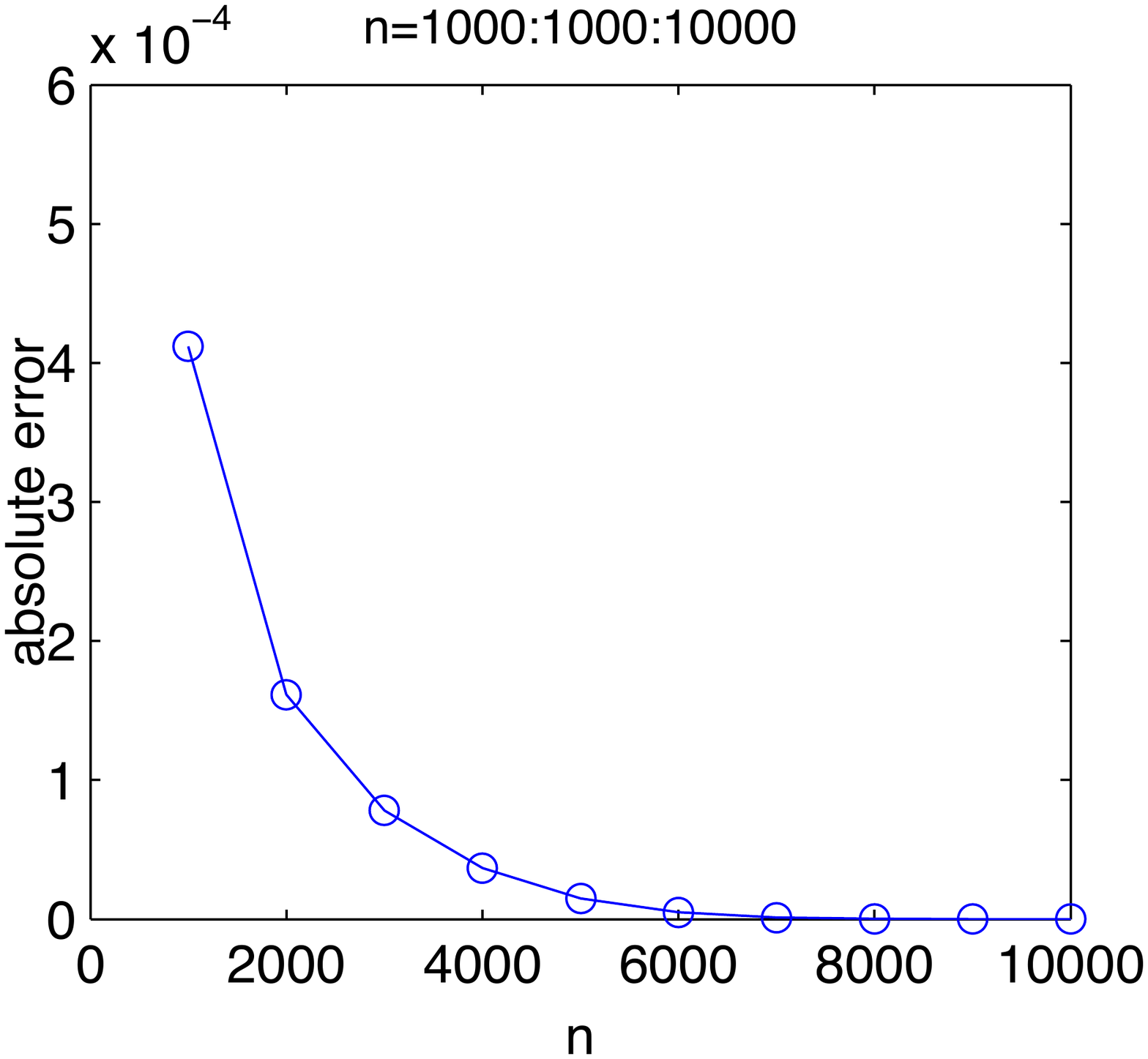}\\
\vspace{-3.5em}
\centerline{\small A:~ $\alpha=100$\hspace{4.2cm} B:~ $\alpha=1000$\hspace{4.2cm} C:~ $\alpha=10000$}
\vspace{-1em}
{\small\caption{The AE of the approximation to the value of the integrals of $G_\alpha$ with fixed $\alpha$
\label{Sec2:Fig5}}}
\end{center}

The numerical phenomena shown in the above numerical example can be well explained by the error representation of the quadrature formula. We now recall the error of the composite quadratures $\mathcal{I}_n^m[G_\alpha]$ for computing $\mathcal{I}_\alpha[f]$. We subdivide $I$ into $n$ equal subintervals and choose the closed Newton-Cotes formula with $m+1$ quadrature nodes to calculate the integrals defined on each subinterval. From the error of the Newton-Cotes formula \cite{Keller} together with  \eqref{Sec2:DerG},
we see that there exists  $\xi\in I$ such that
\begin{equation}\label{Sec2:Error}
\abs{\mathcal{I}_n^m[G_\alpha]-\mathcal{I}_\alpha[f]}
=c_{m}G_\alpha(\xi)
\abs{H_{m+\gamma}(\alpha\xi)}\left(\frac{\alpha}{n}\right)^{m+\gamma},  \ \ \gamma=1\   \text{for odd} \ m \ \text{and}\ \gamma=2\   \text{for even} \ m
\end{equation}
where $c_m$ are defined by
\begin{equation*}
c_m:=\frac{1}{(m+1)!m^{m+2}}
\begin{cases}
J_m, & \text{if} \ m \ \text{is odd},\\
(m(m+2))^{-1}J_m, &  \text{if} \ m \ \text{is even},
\end{cases}
\end{equation*}
with
\begin{equation*}
J_m:=
\begin{cases}
|\int_0^m\prod_{j=0}^m(t-j){\rm d}t|, & \text{if} \ m \ \text{is odd},\\
|\int_0^mt\prod_{j=0}^m(t-j){\rm d}t|, & \text{if} \ m \ \text{is even}.
\end{cases}
\end{equation*}
We introduce a complex number $q_m:=(-{\rm i})^m2^m/\sqrt{\pi}$ and a complex-valued function by 
$$
E_m(t):=\int_{-\infty}^{+\infty}{\rm e}^{-s^2+{\rm i}2ts}s^m{\rm d}s, \ \text{for}\ t\in\mathbb{R},
$$
where ${\rm i}$ denotes the  imaginary unit, see \cite{Abramowitz}. We then write
$$
H_m(t)=q_m{\rm e}^{t^2}E_m(t), \ \text{for}\ t\in\mathbb{R}.
$$
By the definition of $G_\alpha$, we have that
$$
G_\alpha(\xi)H_m(\alpha\xi)=q_mE_m(\alpha\xi).
$$
Substituting the equation above into \eqref{Sec2:Error} yields
$$
\abs{\mathcal{I}_n^m[G_\alpha]-\mathcal{I}_\alpha[f]}=c_m|q_m E_{m+\gamma}(\alpha\xi)|\left(\frac{\alpha}{n}\right)^{m+\gamma}, \ \ \gamma=1\  \text{for odd}\  m \ \text{and}\    \gamma=2  \ \text{for even} \  m.
$$
The error formula above shows that for the quadrature rule to converge, we have to choose $n$ larger than $\alpha$. This will results in large computation costs to compute the integral when $\alpha$ is large. This explain the reason why the trapezoidal rule applied to compute the integral of the Gaussian function does not converge when $\alpha$ is larger than $n$.

We conclude from the analysis above that using traditional quadrature formulas to calculate integrals
involved the Gaussian function becomes expensive to obtain satisfactory numerical results if $\alpha$ is large.
One reason is that the maximum (which occurs near zero) of the high order derivative of $G_\alpha$ is large.
We are required to add more number of quadrature nodes in order to improve accuracy of the numerical integration.
However, the values of the high order derivatives of $G_\alpha$ decay rapidly when moving away from zero  and have high oscillation near zero. Clearly,  adding more quadrature nodes uniformly in the integral interval increases computational costs.
It is therefore important to utilize the property of the Gaussian function to design efficient quadrature formulas for computing these integrals.

To close this section, we describe the main idea to be used in this paper in designing efficient quadrature formulas for computing \eqref{Sec2:Int}. Inspired by the quadrature formulas proposed in \cite{CMX, Xu1, Xu2, Schwab} for weakly singular integrals and in \cite{Xu3, Xu4} for highly oscillatory integrals, we shall develop quadrature formulas for \eqref{Sec2:Int} as follows. We fix a positive integer $n$, independent of the parameter $\alpha$. Design a graded mesh of $n$ subintervals according to $\alpha$. Choose a basic quadrature formula to be used for integration on each of the subintervals associated with the graded mesh. The graded mesh is designed in the way that the resulting integration errors on all the subintervals are approximately equal. We shall show that the accuracy of the proposed quadrature formulas for \eqref{Sec2:Int} increases as $\alpha$ increases and the number of functional evaluations of the integrand used in the quadrature formula is independent of $\alpha$.

In the remaining sections of this paper, we shall complete the following three tasks:

\begin{itemize}
\item [(i)]  Propose a quadrature formula for integrals involved $G_{\alpha,\beta}$ on the interval $[-1,1]$, where the Chebyshev points of the first kind are used as the quadrature nodes. This formula will serve as a basic quadrature formula for developing composite formulas for computing the integrals  \eqref{Sec2:Int}.

\item[(ii)]  Design a partition with a fixed number of subintervals of the integral interval $I$, with the break-points of the partition being distributed according to both the parameter $\alpha$ and the expected accuracy order. Develop composite quadrature formulas for computing the integrals \eqref{Sec2:Int}, one with accuracy of a polynomial order for a smooth factor $f$ having differentiability of a finite order, and another with accuracy of an exponential order for $f$ having infinite differentiability.

\item[(iii)] Develop a method for computing exactly the moments of the basic quadrature formula with the help of the error function.
\end{itemize}

\section{A Basic Quadrature Formula}

We describe in this section a basic quadrature formula for computing the integral
\begin{equation}
\label{Sec3:IntStand} \mathcal{I}_{\alpha,\beta}[f]:=\int_{\hat{I}}f(x)G_{\alpha,\beta}(x){\rm d}x,
\end{equation}
with $\hat{I}:=[-1,1]$,  where $f\in C(\hat{I})$ is independent of $\alpha$ and $G_{\alpha,\beta}$ is defined by \eqref{Sec2:Gauss} for $\alpha>1$. The quadrature formula is designed by interpolating  $f$ at the Chebyshev points of the first kind \cite{Cheney,Davis,Shen,Kuan} and calculating the weights of the resulting formula exactly. The advantage of interpolating at the Chebyshev points is to avoid the well-known Runge phenomenon. The quadrature formula will be used as a basic rule  to design the composite quadrature formula for the integral $\mathcal{I}_\alpha[f]$ defined by \eqref{Sec2:Int} in the following section.

The basic numerical quadrature formula for computing the integral \eqref{Sec3:IntStand} is derived by replacing the function $f$ in the integral by its Lagrange interpolation polynomial.
For a fixed positive integer $m\in\mathbb{N}$, we denote by $p_m$ the Lagrange interpolation polynomial of degree $m$, which interpolates $f$ at the Chebyshev points of the first kind
\begin{equation*}
 x_j:=\cos\left((2j+1)\pi/(2m+2)\right),
\end{equation*}
for $j\in\mathbb{Z}_{m}:=\set{0,1,\ldots,m}$. We then obtain the quadrature formula
\begin{equation}\label{Sec3:QR1}
  \mathcal{Q}^{\alpha,\beta}_m[f]:=\mathcal{I}_{\alpha,\beta}[p_m]
\end{equation}
by replacing the function $f$ in the integral \eqref{Sec3:IntStand} with $p_m$.
We now express the Lagrange polynomial $p_m$ in terms of the Chebyshev polynomials of the first kind $T_n$ for $n\in\mathbb{N}_0$. Namely,
\begin{equation*}
p_m(x)=\mathop{{\sum}'}_{ j\in\mathbb{Z}_m}c_j(f)T_j(x),\ \ \text{for}\ \ x\in I,
\end{equation*}
where the prime denotes a sum whose  first term is halved and for $j\in\mathbb{Z}_m$
\begin{equation*}
c_j (f):=\dfrac{2}{m+1}\sum_{j\in\mathbb{Z}_m}f(x_j)T(x_j).
\end{equation*}
The formula \eqref{Sec3:QR1} is rewritten  as
\begin{equation*}
  \mathcal{Q}^{\alpha,\beta}_m[f]:=\mathop{{\sum}'}_{ j\in\mathbb{Z}_m}c_j(f)\mathcal{I}_{\alpha,\beta}[T_j].
\end{equation*}
Upon introducing the numbers
$$
A_{j,k}:=(-1)^k\frac{(j-k-1)!}{k!(j-2k)!}2^{j-2k-1},
$$
we have that \cite{Alexander}
\begin{equation*}
T_n(x)=n\sum_{k\in\mathbb{Z}_{\lfloor\frac{n}{2}\rfloor}}A_{n,k} x^{n-2k}, ~\text{for}~n\in\mathbb{N},
\end{equation*}
where $\lfloor a \rfloor$ denotes the biggest integer not larger than $a$. Let
$q_j(x):=x^j$,  for  $x\in\hat{I}$,  $j\in\mathbb{Z}_{m}$. We introduce the moments
\begin{equation}\label{Sec3:weight}
w_{\alpha,\beta,j}:=\mathcal{I}_{\alpha,\beta}[q_j], ~~\text{for}~ ~ j\in\mathbb{Z}_{m}.
\end{equation}
Formula \eqref{Sec3:QR1} may be reexpressed in terms of the moments $w_{\alpha,\beta,j}$ as
\begin{equation}\label{QQQs}
\mathcal{Q}^{\alpha,\beta}_m[f]=\mathop{{\sum}'}_{ j\in\mathbb{Z}_m}jc_j(f)\sum_{k\in\mathbb{Z}_{\lfloor\frac{j}{2}\rfloor}}A_{j,k}w_{\alpha,\beta,j-2k}.
\end{equation}
An approximate value of the integral may be computed according to formula \eqref{QQQs} once the moments $w_{\alpha,\beta,j}$ are available. We postpone the computation of the moments $w_{\alpha,\beta,j}$ until Section 5.
Formula \eqref{QQQs} is the basic quadrature formula we shall use in the rest of this paper.

In the next lemma, we present
an estimate of the error of the basic quadrature formula
\begin{equation*}
\mathcal{E}^{\alpha,\beta}_m[f]:=\abs{\mathcal{I}_{\alpha,\beta}[f]-\mathcal{Q}^{\alpha,\beta}_m[f]}.
\end{equation*}
We let  $\norm{\varphi}_\infty:=\max\set{\abs{\varphi(x)}: x\in \Omega}$ for a bounded function $\varphi$ defined in $\Omega$.

\begin{lemma}\label{lem:Sec3}
  If $f\in C^{m+1}(\hat{I})$ for some $m\in\mathbb{N}$ and $\alpha>1$, then
  \begin{equation*}
  \mathcal{E}^{\alpha,\beta}_{m}[f]\leq \dfrac{\sqrt{\pi}}{2^m(m+1)!\alpha}\norm{f^{(m+1)}}_{\infty}.
  \end{equation*}
\end{lemma}
\begin{proof}
By the definition of the error $\mathcal{E}^{\alpha,\beta}_{m}[f]$, we observe that
\begin{equation}\label{basic-estimate}
 \mathcal{E}^{\alpha,\beta}_{m}[f]\leq \norm{f-p_m}_{\infty}\mathcal{I}_{\alpha,\beta}[g]
\end{equation}
with $g(x):=1$ for $x\in\hat{I}$.  According to \cite{Gradshteyn}, we have that
\begin{equation}\label{Sec3:proof}
\mathcal{I}_{\alpha,\beta}[g]\leq\int_{-\infty}^{\infty}G_{\alpha,\beta}(x){\rm d}x= \sqrt{\pi}/\alpha.
\end{equation}
We recall that the error of the Lagrange interpolating of degree $m\in\mathbb{N}$ at the  Chebyshev points of the first  kind for a function $f\in C^{m+1}(\hat{I})$ is bounded by
\begin{equation}\label{Sec3:EChebyInterp}
 \norm{f-p_m}_{\infty}\leq\dfrac{1}{2^m(m+1)!}\norm{f^{(m+1)}}_{\infty}.
\end{equation}
Substituting both \eqref{Sec3:proof} and \eqref{Sec3:EChebyInterp} into the right hand side of \eqref{basic-estimate}  yields the desired estimate.
\end{proof}

To close this section, we comment on the relationship between the parameter $\alpha$
and $\mathcal{I}_\alpha[f]$ defined by \eqref{Sec2:Int} for a bounded function $f$.
From \eqref{Sec3:proof}, we see that  $\mathcal{I}_\alpha[f]$  decays to zero
not slower than $\mathcal{O}(\alpha^{-1})$ as $\alpha\to\infty$. Hence, if we approximate the integral $\mathcal{I}_\alpha[f]$ simply by the number zero, its error is bounded by $\mathcal{O}(\alpha^{-1})$. We may say that any approximation to $\mathcal{I}_\alpha[f]$ is less accurate than simply using the number zero is useless.
We are required to design quadrature formulas for computing $\mathcal{I}_\alpha[f]$ with accuracy better than the zero. 

\section{Quadrature Formulas Based on the Graded Mesh}

We develop in this section composite quadrature formulas for computing the integrals \eqref{Sec2:Int} for the case $f\in C(I)$ is independent of  $\alpha$ and $\alpha>1$ (that is, the standard deviation $c_0$ of the Gaussian function is less than $1/\sqrt{2}$).
The composite quadrature formulas are developed based on graded meshes determined by $\alpha$ and the expected accuracy order of the formulas. Two types of composite quadrature formulas are proposed, one with accuracy of a polynomial order for a smooth factor $f$ having smoothness of a finite order, and the other with accuracy of an exponential order for $f$ having smoothness of the infinite order.

We first motivate the design of the graded mesh of the integral interval $I$. Ideally, the break-points of the graded mesh should be distributed so that the resulting quadrature formula has equal errors on all the subintervals. Specifically, we let $n\in\mathbb{N}$ and suppose that $I$ is partitioned by $0=x_0<x_1<\ldots< x_n=1$ with the break-points $x_j$ to be determined.
By $E_j$ we denote the error of the  quadrature formula on the subinterval $I_j:=[x_{j-1}, x_j]$, $j\in \mathbb{Z}_{n}^+:=\{1,2, \ldots, n\}$. That is,
\begin{equation*}
E_j=\int_{-1}^{1}\abs{h_jf(h_jx/2+(x_j+x_{j+1})/2)/2-p_m(x)}G_\alpha(h_jx/2+(x_j+x_{j+1})/2){\rm d}x,
\end{equation*}
where $h_j:=x_{j}-x_{j-1}$ and $p_m$ is the Lagrange interpolation polynomial  of degree $m\in\mathbb{N}$ at the Chebyshev points of the first  kind to approximate $h_jf(h_jx/2+(x_j+x_{j+1})/2)/2$ for $x\in\hat{I}$.
The best choice of the break-points of the graded-mesh would ensure that the errors $E_j$ are all equal.
However, it is not possible to have the explicit form of $x_j$. We then appeal to conservative upper bounds of $E_j$. From \eqref{Sec3:EChebyInterp}
we obtain upper bounds of the error on all subintervals
\begin{eqnarray}\label{Sec4:0}
E_j\leq \dfrac{G_\alpha(x_j)h_j^{m+2}}{2^{2m+1}(m+1)!}\norm{f^{(m+1)}}_{\infty},
~\text{for}~j\in\mathbb{Z}_{n}^+.
\end{eqnarray}
One strategy to choose
the break-points of the graded-mesh is to solve the following system of equations for $x_j$
\begin{equation}\label{Sec4:000}
G_\alpha(x_j)h_j^{m+2}=G_\alpha(x_{j+1})h_{j+1}^{m+2}, ~\text{for}~j\in\mathbb{Z}_{n-1}^+.
\end{equation}
Again, it is difficult to solve the system explicitly. Instead of solving \eqref{Sec4:000} exactly,
motivated by the idea proposed in \cite{Xu3,Xu4} for computing an oscillatory integral, we propose a strategy to partition the interval $I$. Specifically,
for $n\in\mathbb{N}$ with $n>1$, we partition the interval $I$ with the break-points defined by
\begin{equation}\label{Sec4:Partition}
  x_0=0,~x_j=\alpha^{(j-1)/(n-1)-1}, ~~\text{for}~j\in\mathbb{Z}_n^+.
\end{equation}
This choice of the break-points ensures that equations  \eqref{Sec4:000} are approximately satisfied.

The proposed quadrature formulas are based on the graded-mesh \eqref{Sec4:Partition} with a transformation  
mapping each of the subintervals $[x_{j-1}, x_{j}]$ onto $\hat{I}:=[-1,1]$ and using the basic quadrature formula to compute the resulting integrals on $\hat{I}$. Specifically, we use the affine transformation $x\mapsto h_jx/2+(x_{j-1}+x_j)/2$ to map
$[x_{j-1},x_j]$ onto $\hat {I}$.
The integral \eqref{Sec2:Int} may be written as accordingly
\begin{equation}\label{Sec4:Int}
  \mathcal{I}_\alpha[f]=\sum_{j\in\mathbb{Z}_n^+}\mathcal{I}_{\alpha_j,\beta_j}[f_j],
\end{equation}
where  for $j\in\mathbb{Z}_n^+$, 
\begin{equation}\label{Sec4:alpha}
\alpha_j:=\alpha h_j/2,
\end{equation}
\begin{equation}\label{Sec4:beta}
\beta_j:=-(x_{j-1}+x_j)/h_j,
\end{equation}
and
\begin{equation}\label{Sec4:f}
f_j(x):=\frac{h_j}{2}f\left(\frac{h_j}{2}x+\frac{x_{j-1}+x_j}{2}\right), ~~ \text{for}~ x\in\hat{I}.
\end{equation}
Computing the integral \eqref{Sec2:Int} is then reduced to calculate
the integrals $\mathcal{I}_{\alpha_j,\beta_j}[f_j]$ for $j\in\mathbb{Z}_n^+$ in \eqref{Sec4:Int},
which have the form of \eqref{Sec3:IntStand}. Approximate values of these integrals will be computed by using formula
\eqref{Sec3:QR1}.

We now develop two composite quadrature formulas for computing the integral \eqref{Sec4:Int} by
using the quadrature formula \eqref{Sec3:QR1} to calculate the integrals
$\mathcal{I}_{\alpha_j,\beta_j}[f_j]$ for $j\in\mathbb{Z}_n^+$.  In the first method,
we use the same  number  of quadrature nodes in all of the subintervals. Selecting a fixed positive integer $m$, we use
$\mathcal{Q}^{\alpha_j,\beta_j}_m[f_j]$ defined as in \eqref{Sec3:QR1} to approximate
the integral $\mathcal{I}_{\alpha_j,\beta_j}[f_j]$ for $j\in\mathbb{Z}_n^+$.
The integral \eqref{Sec4:Int} is then approximated by
\begin{equation}\label{alm:Sec4P}
  \mathcal{Q}^{\alpha}_{n,m}[f]:=\sum_{j\in\mathbb{Z}_{n}^+}\mathcal{Q}^{\alpha_j,\beta_j}_m[f_j].
\end{equation}

We next analyze the accuracy order of the quadrature formula  \eqref{alm:Sec4P}.
To this end, in the following lemma we estimate the error
\begin{equation*}
\mathcal{E}^{\alpha_j,\beta_j}_{m}[f_j]:=\abs{\mathcal{I}_{\alpha_j,\beta_j}[f_j]-\mathcal{Q}^{\alpha_j,\beta_j}_{m}[f_j]}, \ \ j\in\mathbb{Z}_n^+.
\end{equation*}

\begin{lemma}\label{lem:Sec4subQR}
  If $f\in C^{m+1}(I)$ and $\alpha>1$, then for $j\in\mathbb{Z}_n^+$
  \begin{equation*}
    \mathcal{E}^{\alpha_j,\beta_j}_{m}[f_j]\leq  \dfrac{\sqrt{\pi}h_j^{m+1}}{2^{2m+1}(m+1)!\alpha}\norm{f^{(m+1)}}_{\infty}.
  \end{equation*}
\end{lemma}
\begin{proof}
 We prove this lemma by applying Lemma \ref{lem:Sec3} to integral $\mathcal{I}_{\alpha_j,\beta_j}[f_j]$.
 According to  the definition \eqref{Sec4:alpha} of $\alpha_j$, we have for $j\in\mathbb{Z}_n^+$ that
 \begin{equation*}
  \mathcal{E}^{\alpha_j,\beta_j}_{m}[f_j]\leq \dfrac{\sqrt{\pi}}{2^{m}(m+1)!\alpha_j}\norm{f_j^{(m+1)}}_{\infty}
 \leq\dfrac{\sqrt{\pi}}{2^{m-1}(m+1)!\alpha h_j}\norm{f_j^{(m+1)}}_{\infty}.
 \end{equation*}
 By the chain rule, we obtain for $x\in\hat{I}$ that
 \begin{equation*}
  f_j^{(m+1)}(x)=\left(\dfrac{h_j}{2}\right)^{m+2}f^{(m+1)}\left(\frac{h_j}{2}x+\frac{x_{j-1}+x_j}{2}\right),
 \end{equation*}
 where $f_j$ is defined by \eqref{Sec4:f}. This together with the inequality above yields the desired estimate.
\end{proof}

We are now ready to provide the error estimate for the quadrature formula \eqref{alm:Sec4P}.
For $m$, $n\in\mathbb{N}$
with $n>1$,
we define
\begin{equation*}
\mathcal{E}^{\alpha}_{n,m}[f]:=\abs{\mathcal{I}_{\alpha}[f]-\mathcal{Q}^{\alpha}_{n,m}[f]},
\end{equation*}
 and denote by $\mathcal{N}(\mathcal{Q}^{\alpha}_{n,m}[f])$ the  number
of the quadrature nodes used in  $\mathcal{Q}^{\alpha}_{n,m}[f]$.

\begin{theorem}\label{thm:Sec4P}
For $\alpha>1$ and $n\in\mathbb{N}$ with $n>1$, let $\eta:=\max\set{1/\alpha,1-\alpha^{-1/(n-1)}}$. If $f\in C^{m+1}(I)$ for some $m\in\mathbb{N}$, then
\begin{equation*}\label{Sec4:P}
\mathcal{E}^{\alpha}_{n,m}[f]\leq  \dfrac{\sqrt{\pi}\eta^{m}}{2^{2m+1}(m+1)!\alpha}\norm{f^{(m+1)}}_{\infty},
\end{equation*}
and  $$\mathcal{N}\left(\mathcal{Q}^{\alpha}_{n,m}[f]\right)=(m+1)n.$$
\end{theorem}
\begin{proof}
The proof of this theorem  is done by applying Lemma \ref{lem:Sec4subQR}  on each of the subintervals. This leads to the estimate that for $j\in\mathbb{Z}_n^+$
\begin{equation}\label{proof4.2}
   \mathcal{E}^{\alpha_j,\beta_j}_{m}[f_j]\leq
  \dfrac{\sqrt{\pi}h_j^{m+1}}{2^{2m+1}(m+1)!\alpha}\norm{f^{(m+1)}}_{\infty}.
\end{equation}
Note that $h_1=x_1-x_0=1/\alpha\leq\eta,$ and for $j\in\mathbb{Z}_n^+$ with $j>1$,
$$
  h_j=x_j-x_{j-1}=x_j\left(1-\alpha^{-1/(n-1)}\right)\leq x_j\eta\leq \eta.
$$
Substituting these bounds into the right-hand side of \eqref{proof4.2} and summing the resulting inequalities over $j\in\mathbb{Z}_n^+$, we obtain the desired estimate for $\mathcal{E}^{\alpha}_{n,m}[f]$.

It remains to estimate the number of the quadrature nodes used in the quadrature formula.
According to the quadrature formula \eqref{alm:Sec4P}, we have that
$$
  \mathcal{N}\left(\mathcal{Q}^{\alpha}_{n,m}[f]\right)=\sum_{j\in\mathbb{Z}_n^+}(m+1)
  = (m+1)n,
$$
proving the desired result. \end{proof}

We remark on the parameter $\eta$ that appears in Theorem \ref{thm:Sec4P}. When $n$ is sufficiently large, for a fixed $\alpha$, the number $\alpha^{-\frac{1}{n-1}}$ is close to 1 and thus, $1-\alpha^{-\frac{1}{n-1}}\leq \frac{1}{\alpha}$. Hence $\eta=\frac{1}{\alpha}$. Therefore, when $n$ is sufficiently large, the accuracy order of the quadrature formula \eqref{alm:Sec4P} is ${\cal O}(1/\alpha^m)$.

We next develop the second method to approximate the integral \eqref{Sec2:Int}.
This method uses the same graded-mesh \eqref{Sec4:Partition} for the interval $I$ but variable numbers of quadrature points in the subintervals. The numbers of the quadrature nodes used in each of the subintervals are chosen so that the resulting quadrature formula has approximately equal errors on all the subintervals.
Specifically, for  $n\in\mathbb{N}$ with $n>1$ and for the partition of $I$ chosen as \eqref{Sec4:Partition}, we let
\begin{equation}\label{Sec4:NumberNode}
m_j:=\left\lceil \frac{n(n-1)}{n+1-j}\right\rceil, ~~\text{for} ~~j\in\mathbb{Z}_{n}^+,
\end{equation}
where $\lceil a\rceil $  denotes the smallest integer not less than $a$.
For each $j\in\mathbb{Z}^{+}_{n}$, we use $\mathcal{Q}^{\alpha_j,\beta_j}_{m_j}[f]$ to approximate $\mathcal{I}_{\alpha_j,\beta_j}[f_j]$.
Integral $\mathcal{I}_\alpha[f]$ defined by \eqref{Sec4:Int}
is then approximated by the quadrature formula
\begin{equation}\label{alm:Sec4E}
\mathcal{Q}^{\alpha}_{n}[f]:=\sum_{j\in\mathbb{Z}_n^+}\mathcal{Q}^{\alpha_j,\beta_j}_{m_j}[f_j].
\end{equation}

We next estimate the error
$$\mathcal{E}^{\alpha}_{n}[f]:=\abs{\mathcal{I}_\alpha[f]-\mathcal{Q}^{\alpha}_{n}[f]}$$
of the quadrature formula $\mathcal{Q}^{\alpha}_{n}[f]$ defined by \eqref{alm:Sec4E}.
To this end, we recall the inequality
\begin{equation}\label{Sec4:Stirling-coro}
n!\geq\sqrt{2\pi n}\left(n/{\rm e}\right)^n, \ \ \mbox{for}\ \ n\in\mathbb{N},
\end{equation}
which is obtained from the Stirling formula \cite{Abramowitz}, and establish the following technical lemma.

\begin{lemma}\label{lem:Sec4Pre}
There exists a positive constant $c$ such that for all $\alpha>1$, for all $n\in\mathbb{N}$ satisfying
\begin{equation}\label{Sec4:C}
(n-1)(\ln{(n+1+{\rm e})}-1)\geq\ln{\alpha},
\end{equation}
and for all  $j\in\mathbb{Z}_n^+$ with $j>1$, the inequality holds
\begin{equation}\label{main-result-of-lemma}
\frac{(\alpha^{1/(n-1)}-1)^{m_j+1}}{(m_j+1)!}\leq c(n+1)^{-1/2}.
\end{equation}
\end{lemma}
\begin{proof}
Since $\alpha>1$, condition \eqref{Sec4:C} implies that  $n>1$.
By inequality \eqref{Sec4:Stirling-coro} with $n:=m_j+1$, there exists a positive constant $c$ such that for all $n\in\mathbb{N}$ with $n>1$ and $j\in\mathbb{Z}_n^+$ with $j>1$,
\begin{equation*}
\frac{1}{(m_j+1)!}\leq c (m_j+1)^{-1/2}\left(\frac{\rm e}{m_j+1}\right)^{m_j+1}.
\end{equation*}
By the definition \eqref{Sec4:NumberNode} of $m_j$,
we see that  $m_j+1\geq n+1$ for $j>1$. Using this result in the right-hand side of the above inequality yields
\begin{equation}\label{QQQQ}
\frac{1}{(m_j+1)!}\leq c (n+1)^{-1/2} \left(\frac{\rm e}{n+1}\right)^{m_j+1}.
\end{equation}
Multiplying both sides of inequality \eqref{QQQQ} by $(\alpha^{1/(n-1)}-1)^{m_j+1}$, we obtain that
\begin{equation}\label{QQQQ00}
\frac{(\alpha^{1/(n-1)}-1)^{m_j+1}}{(m_j+1)!}
\leq c(n+1)^{-1/2}\left(\frac{{\rm e}(\alpha^{1/(n-1)}-1)}{n+1}\right)^{m_j+1}.
\end{equation}
Moreover, condition \eqref{Sec4:C} is equivalent to that
$$
\frac{{\rm e}(\alpha^{1/(n-1)}-1)}{n+1}\leq 1.
$$
Using this inequality in the right-hand side of \eqref{QQQQ00}, we find that there exists a positive
constant $c$ such that for all $\alpha>1$,
$n\in\mathbb{N}$ satisfying \eqref{Sec4:C} and  $j\in\mathbb{Z}_n^+$ with $j>1$, the desired estimate
\eqref{main-result-of-lemma} holds.
\end{proof}

We are now ready to establish the estimate for $\mathcal{E}^{\alpha}_{n}[f]$.
For a function $\phi\in C^{\infty}(\Omega)$, we define
$$\norm{\phi}_{n}:=\max\set{\norm{\phi^{(j)}}_\infty: j\in\mathbb{Z}_n}~\text{for} ~ n\in\mathbb{N}.$$

\begin{theorem}\label{thm:Sec4E}
If $f\in C^{\infty}(I)$, then there exists a positive constant $c$ such that for all $\alpha> 1$ and  $n\in\mathbb{N}$ satisfying \eqref{Sec4:C},
\begin{equation*}\label{Sec4:E}
\mathcal{E}^{\alpha}_{n}[f]\leq c (n+1)^{-1/2}(2\alpha)^{-n-1}\norm{f}_{(n-1)n+1}.
\end{equation*}
For $n\in\mathbb{N}$ with $n>1$, there holds the estimate
$$\mathcal{N}\left(\mathcal{Q}^{\alpha}_{n}[f]\right)\leq n(n-1)\ln{n}+n^2+n.$$
\end{theorem}
\begin{proof}
We establish the error bound of this theorem by estimating the errors
\begin{equation*}
\mathcal{E}^{\alpha_j,\beta_j}_{m_j}[f_j]:=\abs{\mathcal{I}_{\alpha_j,\beta_j}[f_j]-\mathcal{Q}^{\alpha_j,\beta_j}_{m_j}[f_j]}
\end{equation*}
for $j\in\mathbb{Z}_{n}^+$, and then summing them over $j$.
Applying Lemma \ref{lem:Sec4subQR} with $m:=m_1$,  we have that for $j=1$
\begin{equation*}
\mathcal{E}^{\alpha_1,\beta_1}_{m_1}[f_1]\leq\dfrac{\sqrt{\pi} \alpha^{-n-1}}{n!2^{2n-1}}\norm{f^{(n)}}_{\infty}.
\end{equation*}
For $j>1$, by applying  Lemma \ref{lem:Sec4subQR} with 
$$
m:=m_j\ \ \mbox{and}\ \ h_j:=\alpha^{\frac{j-n-1}{n-1}}(\alpha^{1/(n-1)}-1),
$$
we obtain that
\begin{eqnarray}\label{Sec4:proofThmE}
\mathcal{E}^{\alpha_j,\beta_j}_{m_j}[f_j]\leq\dfrac{\sqrt{\pi}}{2^{2m_j+1}}
\dfrac{(\alpha^{1/(n-1)}-1)^{m_j+1}}{(m_j+1)!}\alpha^{\frac{j-n-1}{n-1}(m_j+1)-1}
\norm{f^{(m_j+1)}}_{\infty}.
\end{eqnarray}
Applying the estimate in Lemma \ref{lem:Sec4Pre} to the right hand side of \eqref{Sec4:proofThmE} yields that there exists
a positive constant $c$ such that for all $\alpha>1$, for $n\in\mathbb{N}$ satisfying \eqref{Sec4:C}, and for $j\in\mathbb{Z}_{n}^+$ with $j>1$,
\begin{equation}\label{Sec4:proofThmEEE}
\mathcal{E}^{\alpha_j,\beta_j}_{m_j}[f_j]\leq c\dfrac{(n+1)^{-1/2}}{2^{2m_j+1}}\alpha^{\frac{j-n-1}{n-1}(m_j+1)-1}\norm{f^{(m_j+1)}}_{\infty}.
\end{equation}
Note that for $j\in\mathbb{Z}_{n}^+$ with $j>1$,
$$
\dfrac{j-n-1}{n-1}(m_j+1)-1\leq \dfrac{j-n-1}{n-1}\dfrac{n(n-1)}{n+1-j}-1=-n-1.
$$
Substituting this result into the right hand side of inequality \eqref{Sec4:proofThmEEE}, we obtain the estimate
\begin{equation*}
\mathcal{E}^{\alpha_j,\beta_j}_{m_j}[f_j]\leq c\dfrac{(n+1)^{-1/2}}{2^{2m_j+1}}\alpha^{-n-1}\norm{f^{(m_j+1)}}_{\infty}.
\end{equation*}
Summing up the both sides of the above inequalities over $j\in\mathbb{Z}_{n}^+$,  we observe that
there exists a positive constant $c$ such that for all $\alpha> 1$ and $n\in\mathbb{N}$ satisfying \eqref{Sec4:C}
\begin{eqnarray*}
\mathcal{E}^{\alpha}_{n}[f]
&\leq&\dfrac{\sqrt{\pi}}{n!2^{2n-1}}\alpha^{-n-1}\norm{f^{(n)}}_{\infty}+c\alpha^{-n-1}
       \set{\sum_{j=2}^{n}\dfrac{(n+1)^{-{1}/{2}}}{2^{2m_j+1}}}\norm{f}_{(n-1)n+1}\\
&\leq&\dfrac{\sqrt{\pi}}{n!2^{2n-1}}\alpha^{-n-1}\norm{f^{(n)}}_{\infty}+\dfrac{cn}{2^{2n-1}}
\alpha^{-n-1}(n+1)^{-{1}/{2}}\norm{f}_{(n-1)n+1}\\
&\leq&
\dfrac{4cn}{2^{n}}
(2\alpha)^{-n-1}(n+1)^{-{1}/{2}}\norm{f}_{(n-1)n+1} .
\end{eqnarray*}
Note that $n/2^{n}<1$ for all $n\in\mathbb{N}$. Using this in the last step of the above inequality
leads to the first desired estimate.

It remains to estimate the number $\mathcal{N}\left(\mathcal{Q}^{\alpha}_{n}[f]\right)$ of the quadrature nodes used in the quadrature formula \eqref{alm:Sec4E}. To this end, we note that
\begin{eqnarray*}
\mathcal{N}\left(\mathcal{Q}^{\alpha}_{n}[f]\right)
\leq\sum_{j\in\mathbb{Z}_n^+}\left\{
      \left(\frac{n(n-1)}{n+1-j}+1\right)+1\right\}
\leq 2n+n(n-1)
      \sum_{j\in\mathbb{Z}_n^+}\frac{1}{n+1-j}.
\end{eqnarray*}
For $n\in\mathbb{N}$, by using the following inequality in the above estimate
$$
\sum_{j\in\mathbb{Z}_n^+}\frac{1}{j}\leq \ln{n}+1,
$$
we have that
\begin{eqnarray*}
\mathcal{N}\left(\mathcal{Q}^{\alpha}_{n}[f]\right)  \leq2n+n(n-1)(\ln{n}+1)
  = n(n-1)\ln{n}+n^2+n,
\end{eqnarray*}
which completes the proof.
\end{proof}

From Theorems \ref{thm:Sec4P} and  \ref{thm:Sec4E}, we see that the quadrature formula \eqref{alm:Sec4P} has accuracy of a polynomial order (in terms of $\alpha$) and the quadrature formula \eqref{alm:Sec4E} achieves accuracy of an exponential order (in terms of $\alpha$). The number of the quadrature nodes used in the quadrature formula \eqref{alm:Sec4P} is linear in $n$ and the number of the quadrature nodes used in the quadrature formula \eqref{alm:Sec4E} is quadratic in $n$. Moreover, these numbers for both quadrature formulas are independent of $\alpha$. Specifically, the number of quadrature nodes used in $\mathcal{Q}^{\alpha}_{n,m}[f]$ is $(m+1)n$ for $m$ and that in $\mathcal{Q}^{\alpha}_{n}[f]$ is $n+\sum_{j\in\mathbb{Z}_n^+}m_j$,
where $m_j$ for $j\in\mathbb{Z}_n^+$  are defined by \eqref{Sec4:NumberNode}.

\section{Computation of the Moments}

In this section, we propose a method to compute the exact values of the moments $w_{\alpha,\beta,k}$, for $k\in\mathbb{Z}_{m}$, defined as in \eqref{Sec3:weight}. For this purpose, we first reexpress the moments in terms of integrals involved $G_\alpha$. We then provide a method to calculate these integrals exactly.

We begin with considering the representation  of the moments  $w_{\alpha,\beta,k}$. To this end, we let $k\in\mathbb{N}_0$ and for $j\in\mathbb{Z}_k$, $\alpha>0$ and $b>0$,  we define
\begin{equation} \label{Sec5:Momtent}
   \mathcal{M}_j[\alpha,b]:=\int_0^bx^j{\rm e}^{-\alpha^2x^2}{\rm d}x,
\end{equation}
and recall the binomial coefficients
\begin{equation*}
C_k^j:=\dfrac{k!}{j!(k-j)!}.
\end{equation*}

\begin{proposition}\label{Sec5:thm}
Let $k\in\mathbb{N}_0$.  If  $\alpha>0$ and $\beta\leq 0$, then
  \begin{equation*}
    w_{\alpha,\beta,k}=\sum_{j\in\mathbb{Z}_k}C_k^j\beta^{k-j}\left(\mathcal{M}_j[\alpha,1-\beta]+(-1)^{2k-j}\mathcal{M}_j[\alpha,1+\beta]\right)~\text{for}~\beta>-1,
  \end{equation*}
    \begin{equation*}
    w_{\alpha,\beta,k}=\sum_{j\in\mathbb{Z}_k}C_k^j\beta^{k-j}
    \left(\mathcal{M}_j[\alpha,1-\beta]-\mathcal{M}_j[\alpha,\abs{1+\beta}]\right)~\text{for}~\beta<-1,
  \end{equation*}
    \begin{equation*}
    w_{\alpha, -1, k}=\sum_{j\in\mathbb{Z}_k}C_k^j(-1)^{k-j}\mathcal{M}_j[\alpha,2].
  \end{equation*}
\end{proposition}
\begin{proof}
We first derive forms of $w_{\alpha,\beta,k}$ for different values of $\beta$. From  \eqref{Sec3:weight}, we obtain that
\begin{equation} \label{W_alpha}
 w_{\alpha,\beta,k}=\int_{-1}^1x^ke^{-\alpha^2(x-\beta)^2}{\rm d}x.
\end{equation}
By a change of variables $x-\beta \mapsto x$, we have that
\begin{eqnarray}\label{Sec5:Momtent011}
w_{\alpha,\beta,k}&=&\int_0^{1-\beta}(x+\beta)^k{\rm e}^{-\alpha^2x^2}{\rm d}x
+\int^0_{-1-\beta}(x+\beta)^k{\rm e}^{-\alpha^2x^2}{\rm d}x.
\end{eqnarray}
For $-1<\beta\leq0$, we make a change of variables $x\mapsto -x$ in the second term of the right hand side of equation \eqref{Sec5:Momtent011}, which yields that
\begin{eqnarray}\label{Sec5:Momtent01}
w_{\alpha,\beta,k}&=&\int_0^{1-\beta}(x+\beta)^k{\rm e}^{-\alpha^2x^2}{\rm d}x+
(-1)^k\int_0^{1+\beta}(x-\beta)^k{\rm e}^{-\alpha^2x^2}{\rm d}x.
\end{eqnarray}
For $\beta<-1$, from formula \eqref{Sec5:Momtent011} we have that
\begin{eqnarray}\label{Sec5:Momtent02}
w_{\alpha,\beta,k}
&=& \int_0^{1-\beta}(x+\beta)^k{\rm e}^{-\alpha^2x^2}{\rm d}x-
\int_0^{\abs{1+\beta}}(x+\beta)^k{\rm e}^{-\alpha^2x^2}{\rm d}x.
\end{eqnarray}
For $\beta=-1$, formula \eqref{Sec5:Momtent011} with a substitution $\beta=-1$ gives that
\begin{equation}\label{Sec5:Momtent2}
w_{\alpha,\beta,k}=\int_{0}^{2}(x-1)^k{\rm e}^{-\alpha^2x^2}{\rm d}x.
\end{equation}
We shall apply the binomial formula to $(x+\beta)^k$, $(x-\beta)^k$ in \eqref{Sec5:Momtent01}, $(x+\beta)^k$ in \eqref{Sec5:Momtent02}, and $(x-1)^k$ in \eqref{Sec5:Momtent2} to verify the assertions of this proposition.

We consider the case $-1<\beta\leq0$ only since the other two cases can be similarly handled.
Applying  the binomial formula to the equation  \eqref{Sec5:Momtent01}, we obtain  that
\begin{eqnarray*}
  w_{\alpha,\beta,k}=\sum_{j\in\mathbb{Z}_k}C_k^j\beta^{k-j}\int_0^{1-\beta}x^j {\rm e}^{-\alpha^2x^2}{\rm d}x+(-1)^{k}\sum_{j\in\mathbb{Z}_k}C_k^j(-\beta)^{k-j}\int_0^{1+\beta}x^j {\rm e}^{-\alpha^2x^2}{\rm d}x.
\end{eqnarray*}
 This together with the definition \eqref{Sec5:Momtent} yields the first result.
\end{proof}

We next describe a  method to compute $\mathcal{M}_k[\alpha,b]$  defined as in \eqref{Sec5:Momtent} for  $k\in\mathbb{N}_0$,
$\alpha>0$ and $b>0$.
We first show an auxiliary equality. To this end,
we denote by $\mathcal{D}_\alpha$ the differentiation operator
with respect to $\alpha$.

\begin{lemma}\label{lem:Sec5}
If $j\in\mathbb{N}_0$, then
\begin{equation*}
\mathcal{D}_\alpha M_j[\sqrt{\alpha}, b]=- M_{j+2}[\sqrt{\alpha}, b].
\end{equation*}
\end{lemma}
\begin{proof}
Note that
\begin{equation} \label{Sec5:MomtentS}
\mathcal{M}_j[\sqrt{\alpha},b]=\int_0^bx^j{\rm e}^{-\alpha x^2}{\rm d}x,
\end{equation}
for $j\in\mathbb{N}_0$.
Differentiating $\mathcal{M}_j[\sqrt{\alpha},b]$ with respect to  $\alpha$, we obtain that
\begin{equation*}
 \mathcal{D}_\alpha M_j[\sqrt{\alpha}, b]=-\int_0^bx^{j+2}{\rm e}^{-\alpha x^2}{\rm d}x.
\end{equation*}
This together with  \eqref{Sec5:MomtentS} by replacing $j $ with $j+2$ yields the desired formula.
\end{proof}

According to Lemma \ref{lem:Sec5}, the computation of $\mathcal{M}_k[\alpha,b]$ may be done by consider two separate cases when $k$ is an even number and when $k$ is an odd number. We shall use the Leibniz formula for the $k$-th derivative of the product of two functions for $k\in\mathbb{N}$,
\begin{equation}\label{Sec5:Leibniz}
\left(\phi\psi\right)^{(k)}=\sum_{j\in\mathbb{Z}_k} C_k^j\phi^{(j)}\psi^{(k-j)}.
\end{equation}

We first consider the case when $k$ is an odd number.

\begin{proposition}\label{lem:Sec5oddM}
If $k=2j+1$ for $j\in\mathbb{N}_0$, then
\begin{equation*}
  \mathcal{M}_{k}[\alpha,b]=\dfrac{j!}{2}\left(1-{\rm e}^{-b^2\alpha^2 }\right)\alpha^{-2-2j}-
  \dfrac{j!}{2}{\rm e}^{- b^2\alpha^2}\sum_{l\in\mathbb{Z}_j^+}{b^{2l}\alpha^{-2-2j+2l}}/{l!}.
\end{equation*}
\end{proposition}
\begin{proof}
We first derive the form of $\mathcal{M}_{k}[\sqrt{\alpha},b]$ for $k=2j+1$ with $j\in\mathbb{N}_0$.
Repeatedly using Lemma \ref{lem:Sec5}, we obtain that
$$
\mathcal{M}_{k}[\sqrt{\alpha},b]=(-1)^j\mathcal{D}_\alpha^j \mathcal{M}_{1}[\sqrt{\alpha},b].
$$
We shall make use of the above formula by differentiating $\mathcal{M}_1[\sqrt{\alpha},b]$ with respect to $\alpha$. By a direct computation, we have that
\begin{eqnarray*}
   \mathcal{M}_{1}[\sqrt{\alpha},b]= \left(1-{\rm e}^{-b^2\alpha }\right)/(2\alpha).
\end{eqnarray*}
Hence, we find that
\begin{equation}\label{Sec5:proofLemOdd}
\mathcal{M}_{k}[\sqrt{\alpha},b]=\frac{(-1)^{j}}{2}\mathcal{D}_\alpha^{j}\left[\left(1-{\rm e}^{-b^2\alpha }\right)/\alpha\right].
\end{equation}
Note that  for $l\in\mathbb{Z}_j^+$
$\mathcal{D}_\alpha^{l}(1-{\rm e}^{-b^2\alpha })=(-1)^{l-1}b^{2l}{\rm e}^{-b^2\alpha }$ and
$\mathcal{D}_\alpha^{l}(1/\alpha)=(-1)^ll!\alpha^{-l-1}.$
Applying the Leibniz formula \eqref{Sec5:Leibniz} with these two derivative formulas, we obtain that
\begin{equation*}
\mathcal{D}_\alpha^{j}\left[\left(1-{\rm e}^{-b^2\alpha }\right)/{\alpha}\right]
=(-1)^jj!\left(1-{\rm e}^{-b^2\alpha }\right)\alpha^{-1-j}+(-1)^{j-1}{\rm e}^{- b^2\alpha}\sum_{l\in\mathbb{Z}_j^+}C_j^l(j-l)!b^{2l}\alpha^{-1-j+l}.
\end{equation*}
Substituting this equality into \eqref{Sec5:proofLemOdd} yields that
\begin{equation*}
     \mathcal{M}_{k}[\sqrt{\alpha},b]=\dfrac{j!}{2}\left(1-{\rm e}^{-b^2\alpha }\right)\alpha^{-1-j}-
     \dfrac{j!}{2}{\rm e}^{- b^2\alpha}\sum_{l\in\mathbb{Z}_j^+}\dfrac{b^{2l}\alpha^{-1-j+l}}{l!}.
\end{equation*}
In the above equation, we replace $\sqrt{\alpha}$ by $\alpha$ and obtain
the desired result of this proposition.
\end{proof}

We next consider the case when $k$ is an even number.
We shall use the notation of the error function \cite{Abramowitz} defined by
\begin{equation}\label{Sec5:erf}
  {\rm erf}(x):=\dfrac{2}{\sqrt{\pi}}\int_0^x{\rm e}^{-s^2}{\rm d}s,~\text{for}~x\in\mathbb{R}^+.
\end{equation}
For more information about the error function, see for instance \cite{Lebedev, Temme}. We shall use $\Gamma$ to denote the Gamma function and recall that for a constant $c\in(-1,1)$ and $n\in\mathbb{N}$,
\begin{equation*}
\frac{\Gamma(c+n)}{\Gamma(c)}=c(c+1)\cdots(c+n-1).
\end{equation*}
In particular, for $c=1/2$, we have that
$
\Gamma(1/2)=\sqrt{\pi}.
$

\begin{proposition}\label{lem:Sec5evenM}
If $k=2j$ for $j\in\mathbb{N}_0$, then
\begin{eqnarray*}
\mathcal{M}_{k}[\alpha,b]=\dfrac{\Gamma(\frac{1}{2}+j)}{2}{\rm erf}(b\alpha)\alpha^{-1-2j}
-\frac{1}{2\pi}
\sum_{l\in\mathbb{Z}_j^+}\sum_{i\in\mathbb{Z}_{l-1}}
\dfrac{j!\Gamma(l-i-\frac{1}{2})\Gamma(\frac{1}{2}+j-l)b^{2i+1}}{(j-l)!(l-i-1)!i!l}{\rm e}^{-b^2\alpha^2}\alpha^{-2j+2i}.
\end{eqnarray*}
\end{proposition}
\begin{proof}
The proof of this result is similar to that of Proposition \ref{lem:Sec5oddM}.
We first derive the form of $\mathcal{M}_{k}[\sqrt{\alpha},b]$ for $k=2j$ with $j\in\mathbb{N}_0$.
According to Lemma \ref{lem:Sec5}, we have that
$$
\mathcal{M}_{k}[\sqrt{\alpha},b]=(-1)^j\mathcal{D}_\alpha^j \mathcal{M}_{0}[\sqrt{\alpha},b].
$$
A direct computation leads to
\begin{eqnarray*}
\mathcal{M}_{0}[\sqrt{\alpha},b]=\int_0^b{\rm e}^{-\alpha x^2}{\rm d}x
 =\frac{\sqrt{\pi}}{2\sqrt{\alpha}}{\rm erf}(b\sqrt{\alpha}).
\end{eqnarray*}
Differentiating the equation above $j$ times with respect to  $\alpha$  yields that
\begin{equation}\label{Sec5:proofLemEven}
 \mathcal{M}_{k}[\sqrt{\alpha},b]=\frac{(-1)^{j}\sqrt{\pi}}{2}
  \mathcal{D}_\alpha^{j}\left[\frac{{\rm erf}(b\sqrt{\alpha})}{\sqrt{\alpha}}\right].
\end{equation}
Using the definition \eqref{Sec5:erf}, we have that
\begin{equation}\label{DDD}
 \mathcal{D}_\alpha\left[{\rm erf}(b\sqrt{\alpha})\right]=\frac{b}{\sqrt{\pi}}\frac{{\rm e}^{-b^2\alpha}}{\sqrt{\alpha}}.
\end{equation}
Note that
\begin{equation}\label{Sec5:proof1LemEven}
 \mathcal{D}_\alpha^{l}\left(\alpha^{-\frac{1}{2}}\right)=(-1)^l\dfrac{\Gamma(\frac{1}{2}+l)}{\Gamma(\frac{1}{2})}\alpha^{-\frac{1}{2}-l} ~\text{for}~l\in\mathbb{Z}_j.
\end{equation}
Using formula \eqref{DDD} with the Leibniz formula \eqref{Sec5:Leibniz} and the equation above yields that for $l\in\mathbb{Z}_j^+$
\begin{equation*}
  \mathcal{D}_\alpha^{l}\left({\rm erf}(b\sqrt{\alpha})\right)=\dfrac{(-1)^{l-1}}{\sqrt{\pi}}
  \sum_{i\in\mathbb{Z}_{l-1}}C_{l-1}^i\dfrac{\Gamma(l-i-\frac{1}{2})b^{2i+1}}{\Gamma(\frac{1}{2})}{\rm e}^{-b^2\alpha}\alpha^{\frac{1}{2}-l+i}.
\end{equation*}
Applying the Leibniz formula \eqref{Sec5:Leibniz} with the equation above and \eqref{Sec5:proof1LemEven}, we obtain that
\begin{eqnarray*}
\mathcal{D}_\alpha^{j}\left[\frac{{\rm erf}(b\sqrt{\alpha})}{\sqrt{\alpha}}\right]
&= &\dfrac{\Gamma(\frac{1}{2}+j)}{\Gamma(\frac{1}{2})}(-1)^j{\rm erf}(b\sqrt{\alpha})\alpha^{-\frac{1}{2}-j}\\
& &
+\frac{(-1)^{j-1}}{\sqrt{\pi}}
\sum_{l\in\mathbb{Z}_j^+}\sum_{i\in\mathbb{Z}_{l-1}}
\dfrac{j!\Gamma(l-i-\frac{1}{2})\Gamma(\frac{1}{2}+j-l)b^{2i+1}}{(j-l)!(l-i-1)!i!l(\Gamma(\frac{1}{2}))^2}{\rm e}^{-b^2\alpha}\alpha^{-j+i}.
\end{eqnarray*}
Substituting the equation above and $\Gamma(\frac{1}{2})=\sqrt{\pi}$ into \eqref{Sec5:proofLemEven} gives  that
\begin{eqnarray*}
\mathcal{M}_{k}[\sqrt{\alpha},b]=\dfrac{\Gamma(\frac{1}{2}+j)}{2}{\rm erf}(b\sqrt{\alpha})\alpha^{-\frac{1}{2}-j}
-\dfrac{1}{2\pi}
\sum_{l\in\mathbb{Z}_j^+}\sum_{i\in\mathbb{Z}_{l-1}}
\dfrac{j!\Gamma(l-i-\frac{1}{2})\Gamma(\frac{1}{2}+j-l)b^{2i+1}}{(j-l)!(l-i-1)!i!l}{\rm e}^{-b^2\alpha}\alpha^{-j+i}.
\end{eqnarray*}
In the equation above replacing $\sqrt{\alpha}$ by $\alpha$, we obtain
the desired result of this proposition.
\end{proof}

Combining formulas in Propositions 5.1, 5.3 and 5.4 provide a method to compute the exact values of the moments $w_{\alpha,\beta,k}$.

\section{Numerical experiments}

In this section, we carry out three numerical experiments to confirm the convergence estimates of the two quadrature formulas for \eqref{Sec2:Int} established in Section 4. Specifically, we verify  the relative errors (RE) and the convergence orders  (Order) of the two quadrature formulas.

The numerical results presented below were all obtained by using Matlab in a modest desktop (a Core 2 Quad with 4GB of Ram memory).
The error function \eqref{Sec5:erf} was computed by using {\it erf} of Matlab. Numerical methods for calculating the error function may be found in \cite{Cody, Chevillard, Strecok}.

We present in the first example the numerical results of the quadrature formula \eqref{alm:Sec4P} which we denote by QuadP.
The numerical convergence order of QuadP is computed by using the formula
\begin{equation*}
  {\rm Order}:=-\ln\left(\mathcal{E}_{n,m}^\alpha[f]\right)/\ln(\alpha),~\text{for}~\alpha>1,
\end{equation*}
for $m$, $n\in\mathbb{N}$ with $n>1$.

\begin{example}{\rm
This example is to verify  the  computational efficiency of the quadrature formula $\mathcal{Q}^{\alpha}_{n,m}[f]$ defined as in \eqref{alm:Sec4P}.
We consider the function $f(x):=x^2$ for $x\in I$.
The exact value of the corresponding integral is
$$\mathcal{I}_\alpha[f]=\left(\sqrt{\pi}{\rm erf}(\alpha)/2-\alpha{\rm e}^{-\alpha^2}\right)/(2\alpha^3).$$
}\end{example}

\vspace{-1em}
\begin{center}
\footnotesize
\makeatletter
\def\@captype{table}
\makeatother
{\caption{Relative errors and convergence orders of QuadP for $f(x):=x^2$}
\vspace{0.2cm}
\label{Sec6:Tab1Eg1}}{
\begin{tabular}{c|c|c|c|c|c|c|c}
\hline
\multirow{2}{*}{$ \alpha$}&\multirow{2}{*}{$c_0$}
&\multicolumn{2}{c|}{ $n=5$ }&\multicolumn{2}{c|}{$n=10$}
&\multicolumn{2}{c}{ $n=15$ }\\
\cline{3-8}
&& RE&Order&RE&Order&RE&Order\\
\hline
10       & 7.07e-2 &3.56e-14&16.80&2.93e-15&17.89&2.59e-14&16.94\\
50       & 1.41e-2 &1.55e-14&11.34&1.72e-14&11.31&3.46e-15&11.72\\
100     & 7.07e-3 &3.46e-15&10.41&3.27e-14&9.92 &3.94e-15&10.38\\
500     & 1.41e-3 &1.60e-13&7.87 &1.41e-14&8.26 &3.69e-14&8.11 \\
1000   & 7.07e-4 &2.74e-13&7.31 &1.60e-14&7.72 &1.38e-14&7.38 \\
5000   & 1.41e-4 &4.78e-15&6.97 &1.22e-14&6.86 &9.80e-15&6.88 \\
10000 & 7.07e-5 &1.51e-14&6.55 &1.25e-15&6.81 &3.53e-15&6.70 \\
\hline
\end{tabular}}
\end{center}

We list in Table \ref{Sec6:Tab1Eg1} the relative errors and the numerical convergence orders  of $\mathcal{Q}^{\alpha}_{n,4}[f]$ for different values of $\alpha$ and different choices of $n$.
The relative errors of the quadrature formula $\mathcal{Q}^{\alpha}_{n,4}[f]$ are depicted in Figure \ref{Sec6:Fig1Eg1}(A) for fixed $\alpha=50$ with $n$ changing from $2$ to $100$ and in Figure \ref{Sec6:Fig1Eg1}(B) for fixed $n=20$ with $\alpha$ changing from $10$ to $500$.
Table \ref{Sec6:Tab1Eg1} confirms the  computational efficiency of the quadrature formula $\mathcal{Q}^{\alpha}_{n,4}[f]$,  since the convergence orders of $\mathcal{Q}^{\alpha}_{n,4}[f]$ are higher than the asymptotic order of the values of the integrals decaying, which is $\mathcal{O}(\alpha^{-3})$ as $\alpha$ tends to infinity. From Figure \ref{Sec6:Fig1Eg1}, we observe that for a fixed $\alpha$ the relative errors of the quadrature formula change  modestly as $n$ grows and for a fixed $n$ they also
change modestly as $\alpha$ grows. This demonstrates that the quadrature formula \eqref{alm:Sec4P} is efficiency and has accuracy of a high order.

\vspace{-2em}
\begin{center}
\makeatletter
\def\@captype{figure}
\makeatother
\includegraphics[width=7.0cm, height=8.5cm]{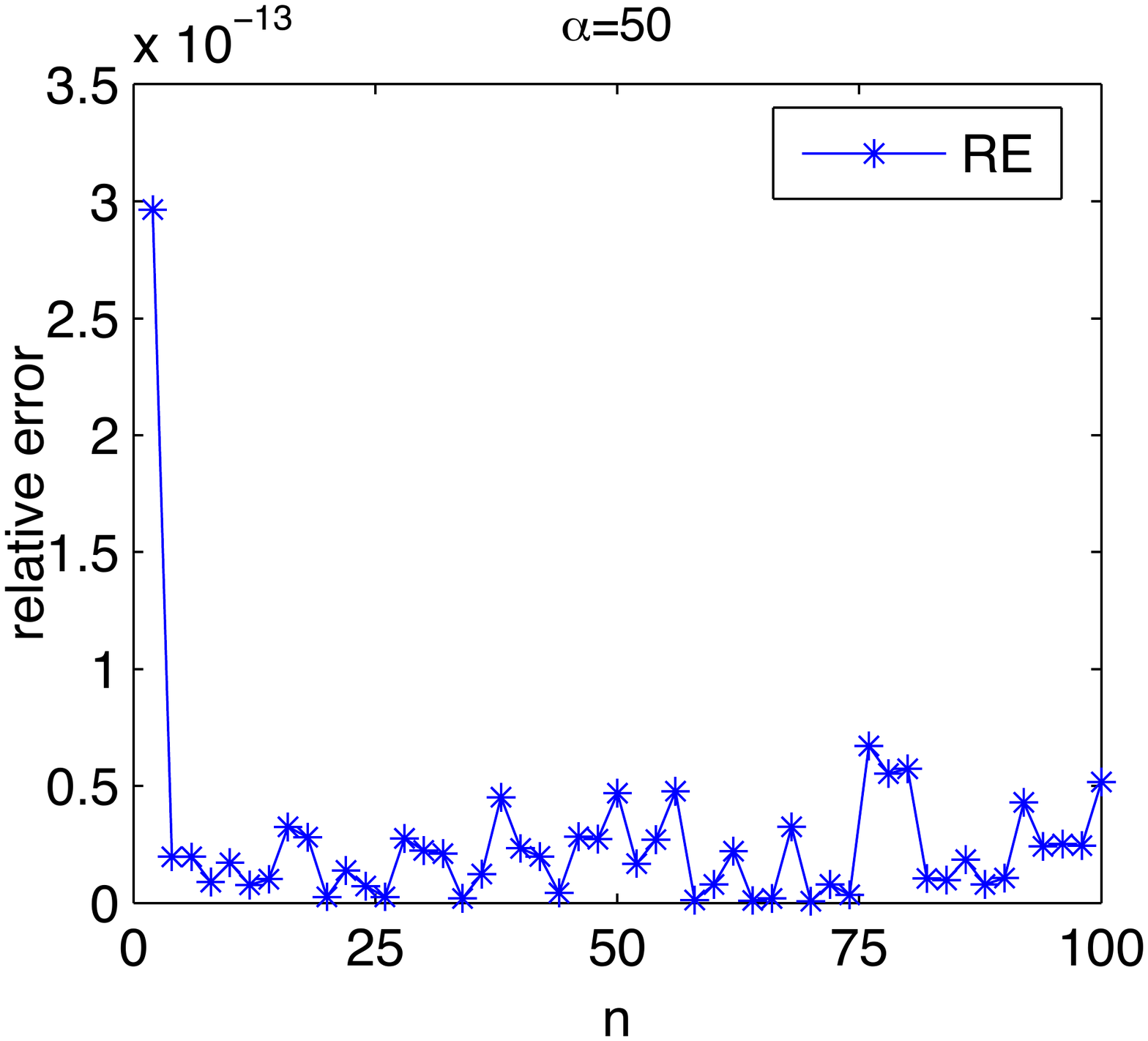}\ \hspace{1cm} \
\includegraphics[width=7.0cm, height=8.5cm]{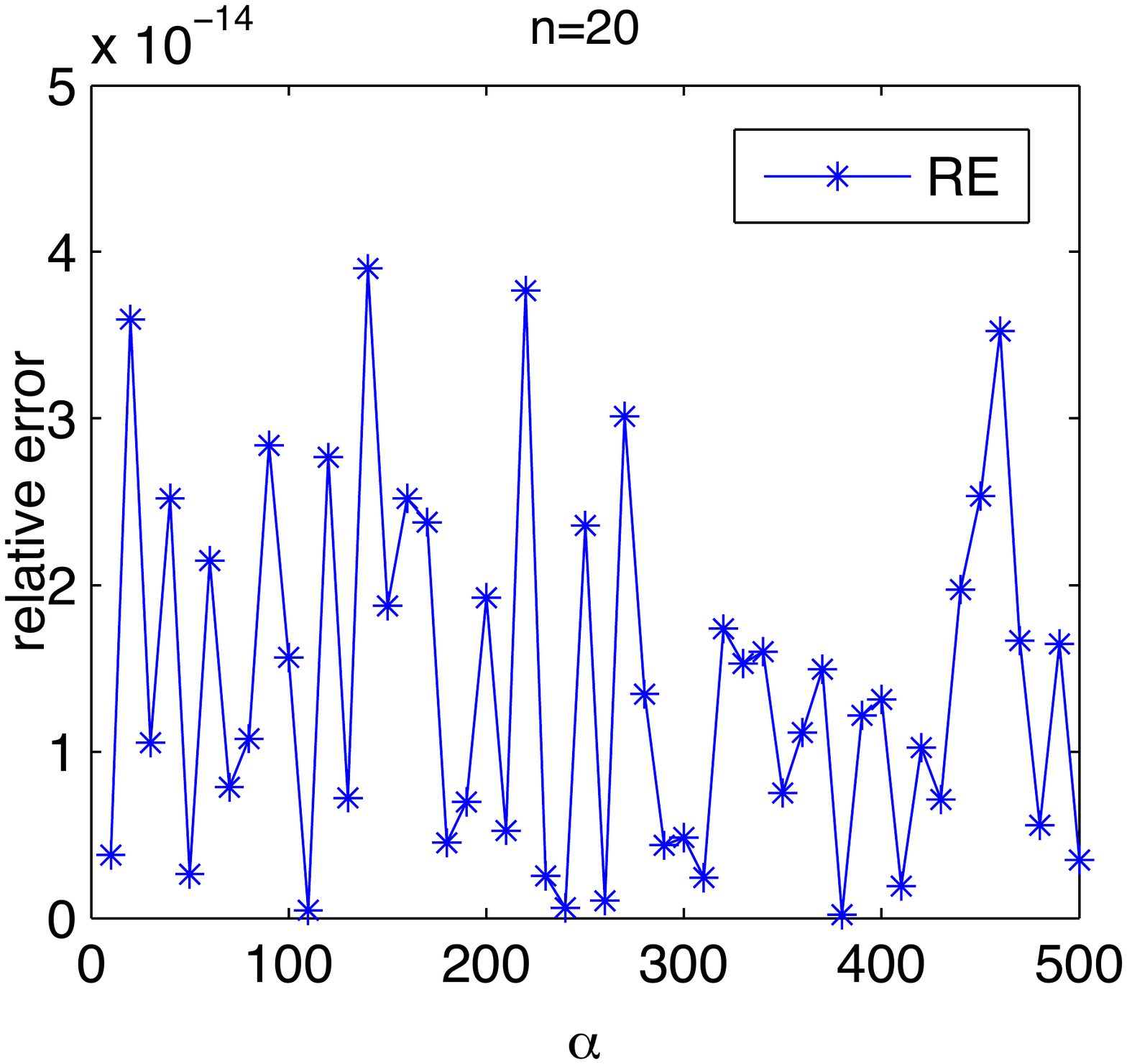}\\
\vspace{-4.5em}
\centerline{\small  A: ~Fixed $\alpha$ ($c_0$=1.14e-2 )~\hspace{7cm}\hspace{-6em} B: ~Fixed $n$}
\vspace{-1em}
{\small\caption{Relative errors of the approximation to the values of the integrals with $f(x):=x^2$}
\label{Sec6:Fig1Eg1}}
\end{center}

In  Table \ref{Sec6:Tab2Eg1}, we compare the relative errors  of the formula $\mathcal{Q}^{\alpha}_{n,2}[f]$ (Graded Mesh) with that of the standard composite  formula (Uniform). The standard composite  formula is constructed by subdividing the integration interval into $n$ equal subintervals and using the Simpson rule to calculate the integrals on each of the resulting subintervals. We conclude that the accuracy order of $\mathcal{Q}^{\alpha}_{n,2}[f]$ is much higher than the composite Simpson formula with a uniform partition.

\begin{center}
\footnotesize
\makeatletter
\def\@captype{table}
\makeatother
{\caption{Relative errors  of QuadP for $f(x):=x^2$ with different partitions}
\vspace{0.2cm}
\label{Sec6:Tab2Eg1}}{
\begin{tabular}{c|c|c|c|c|c|c}
\hline
\multirow{2}{*}{$ n$}
&\multicolumn{2}{c|}{$\alpha$=20 ($c_0$=3.54e-2)}
&\multicolumn{2}{c|}{$\alpha$=30  ($c_0$=2.36e-2)}
&\multicolumn{2}{c} {$\alpha$=40  ($c_0$=1.77e-2) }\\
\cline{2-7}
&Uniform&Graded Mesh&Uniform&Graded Mesh&Uniform&Graded Mesh\\
\hline
n=5    &5.59e-1&1.05e-14 &9.90e-1&7.64e-15&1.00e-0&8.44e-15\\
n=10  &2.21e-1&2.47e-14&7.29e-2&1.71e-14&5.59e-1&5.33e-14\\
n=20 &6.46e-4 &3.61e-14&6.45e-2&1.05e-14&2.21e-1&2.50e-14\\
\hline
\end{tabular}}
\end{center}

We next show  the numerical results of the quadrature formula \eqref{alm:Sec4E}  which we
 denote by QuadE. As in the first example, we compute the numerical convergence order of QuadE by using the formula
\begin{equation*}
  {\rm Order}:=-\ln\left(\mathcal{E}_{n}^\alpha[f]\right)/\ln(2\alpha),~\text{for}~\alpha>1,
\end{equation*}
for $n\in\mathbb{N}$ with $n>1$.

\begin{example}{\rm
We verify in this example the theoretical estimate presented in Theorem \ref{thm:Sec4E} for the quadrature formula $\mathcal{Q}^{\alpha}_{n}[f]$ defined by \eqref{alm:Sec4E}.
We consider the function $f(x):=\exp\set{-x^2}$ for $x\in I$.
The exact value of the corresponding integral is
$$\mathcal{I}_\alpha[f]=\sqrt{\pi}{\rm erf}(\sqrt{\alpha^2+1})/(2\sqrt{\alpha^2+1}).$$
}\end{example}

\begin{center}
\footnotesize
\makeatletter
\def\@captype{table}
\makeatother
{\caption{Relative errors and convergence orders of QuadE for $f(x):=\exp\set{-x^2}$}
\vspace{0.2cm}
\label{Sec6:Tab1Eg2}}{
\begin{tabular}{c|c|c|c|c|c|c|c}
\hline
\multirow{2}{*}{$ \alpha$}&\multirow{2}{*}{$ c_0$}
&\multicolumn{2}{c|}{ $n=3$ }&\multicolumn{2}{c|}{$n=4$}
&\multicolumn{2}{c}{ $n=5$ }\\
\cline{3-8}
& &RE&Order&RE&Order&RE&Order\\
\hline
20        & 3.54e-2 &1.37e-7  &5.13&1.12e-9  &6.43&1.14e-13&8.92\\
80        & 8.84e-3 &6.16e-9  &4.61&4.65e-12&6.03&4.70e-16&7.84\\
160      & 4.42e-3 &5.38e-9  &4.20&2.99e-13&5.90&3.13e-16&7.09\\
200      & 3.54e-3 &4.26e-9  &4.12&1.23e-13&5.87&1.96e-16&6.94 \\
800      & 8.84e-4 &5.50e-10&3.81&7.83e-16&5.64&1.96e-16&5.83\\
2000    & 3.54e-4 &1.10e-10&3.70&4.89e-16&5.18&1.22e-16&5.35 \\
\hline
\end{tabular}}
\end{center}

Numerical results of this example are reported in Tables \ref{Sec6:Tab1Eg2}--\ref{Sec6:Tab2Eg2}
and Figure \ref{Sec6:Fig1Eg2}.
The relative errors  and the order of convergence of $\mathcal{Q}^{\alpha}_{n}[f]$ for different values of $\alpha$ and choices of $n=3,4,5$ are listed in Table \ref{Sec6:Tab1Eg2}. We plot the relative errors of $\mathcal{Q}^{\alpha}_{n}[f]$ for fixed $\alpha=600$ in Figure \ref{Sec6:Fig1Eg2} (A) with $n$ ranging from $4$ to $12$.
The absolute errors (AE) of $\mathcal{Q}^{\alpha}_{n}[f]$ multiplied by $(2\alpha)^{n+1}$ are depicted
in Figure \ref{Sec6:Fig1Eg2} (B) for $n=4$, where we choose $\alpha$ ranging from $300$ to $400$.
From Table \ref{Sec6:Tab1Eg2} and Figure \ref{Sec6:Fig1Eg2} (A), we observe that for
a fixed $\alpha$, the accuracy of the quadrature formula improves as $n$ grows and
for a fixed $n$, it also improves as $\alpha$ grows. From Figure \ref{Sec6:Fig1Eg2} (B) we see that the asymptotic order of convergence of $\mathcal{Q}^{\alpha}_{n}[f]$ is $\mathcal{O}((2\alpha)^{-5})$ for $n=4$, which concurs with the theoretical estimate given in  Theorem \ref{thm:Sec4E}. In Table \ref{Sec6:Tab2Eg2}, we compare the RE of the formula $\mathcal{Q}^{\alpha}_{n}[f]$ (Graded Mesh) with that of the composite Simpson formulas  (Uniform) having  $n$ uniform subintervals.
We see that the accuracy of $\mathcal{Q}^{\alpha}_{n}[f]$ is much higher (with fewer quadrature nodes) than that of the composite Simpson rule. The formulas defined  by \eqref{alm:Sec4E} enhance the approximation accuracy and as well as reduce the computational complexity.

\vspace{-2em}
\begin{center}
\makeatletter
\def\@captype{figure}
\makeatother
\includegraphics[width=7.0cm, height=8.5cm]{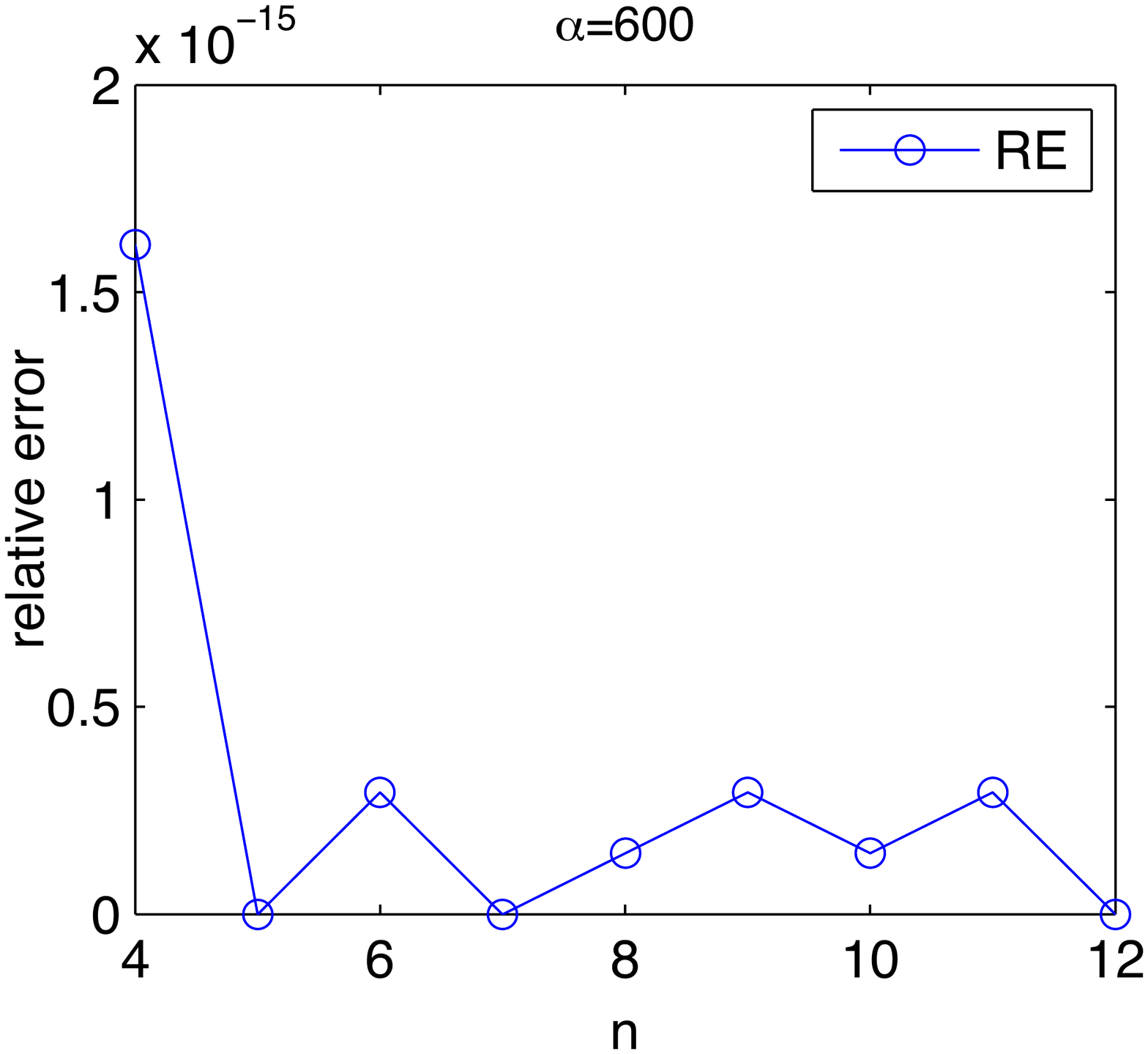}\ \hspace{2cm} \
\includegraphics[width=7.0cm, height=8.5cm]{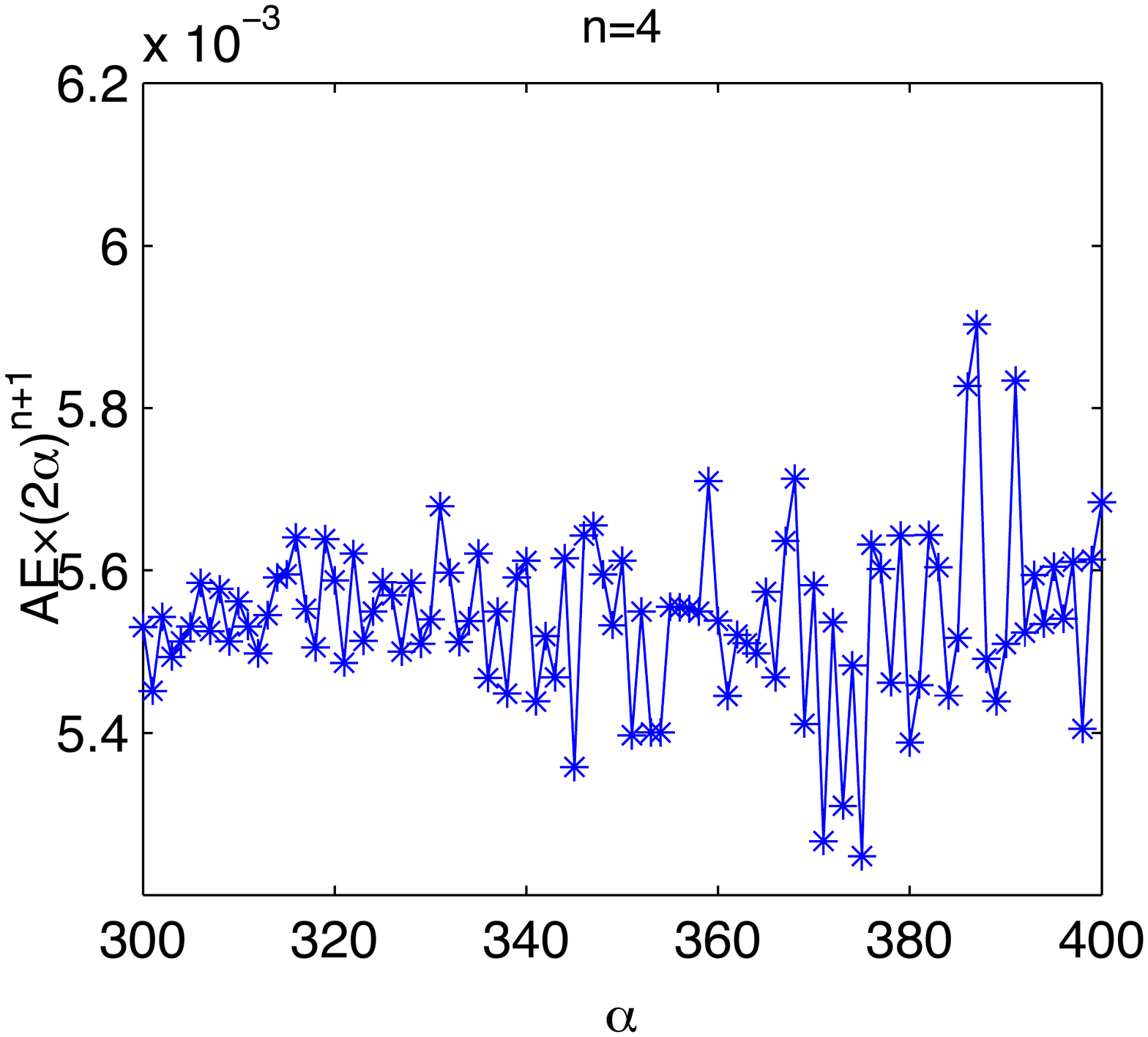}\\
\vspace{-4.5em}
\centerline{\small A: ~RE for fixed $\alpha$ ~\hspace{6cm} B: ~$(2\alpha)^5\mathcal{E}^{\alpha}_{4}[f]$ }
\vspace{-1em}
{\small\caption{Relative errors of $\mathcal{Q}^{\alpha}_{n}[f]$ with fixed $\alpha$ and $\mathcal{E}^{\alpha}_{4}[f]$ scaled by $\alpha^5$ for $f(x):=\exp\set{-x^2}$}
\label{Sec6:Fig1Eg2}}
\end{center}

\vspace{-1em}
\begin{center}
\footnotesize
\makeatletter
\def\@captype{table}
\makeatother
{\caption{Relative errors  of QuadE for $f(x):=\exp\set{-x^2}$ with different partitions}
\vspace{0.2cm}
\label{Sec6:Tab2Eg2}}{
\begin{tabular}{c|c|c|c|c|c|c|c|c|c|c|c}
\hline
\multicolumn{4}{c|}{ $\alpha=30$ ($c_0$=2.36e-2) }&
\multicolumn{4}{c|}{$\alpha=50$  ($c_0$=1.41e-2) }
&\multicolumn{4}{c}{ $\alpha=100$ ($c_0$=7.07e-3) }\\
\hline
\multicolumn{2}{c|}{Uniform}&\multicolumn{2}{c|}{ Graded Mesh}
&\multicolumn{2}{c|}{ Uniform }&\multicolumn{2}{c|}{ Graded Mesh }
&\multicolumn{2}{c}{ Uniform }&\multicolumn{2}{c}{Graded Mesh}\\
\hline
RE  & $\mathcal{N}$ & RE  &$ \mathcal{N}$& RE  & $\mathcal{N}$ & RE  & $\mathcal{N}$  &RE  &$ \mathcal{N}
$&RE  & $ \mathcal{N}$\\
\hline
  1.61e-1 &15 &5.54e-8  & 14& 3.44e-1&15&5.84e-9  & 14&1.69   &15&6.97e-9   & 14\\
  5.66e-2 &31 &2.34e-10& 29& 2.17e-1 &31&2.97e-11& 29&2.54e-1&31&1.93e-12& 29\\
\hline
\end{tabular}}
\end{center}

In the last example, we test the computational efficiency of the quadrature formulas \eqref{alm:Sec4P} and \eqref{alm:Sec4E} for a piecewise continuous function.

\begin{example}{\rm
We consider in this example the function
\begin{equation}\label{Sec6:fun}
f(x):=
\begin{cases}
  1,& \text{for}~x\in[0, 1/2],\\
  1/2, &\text{for}~x\in(1/2,1].
\end{cases}
\end{equation}
The exact value of the corresponding integral is
$$
\mathcal{I}_\alpha[f]=\sqrt{\pi}\left({\rm erf}(\alpha)+{\rm erf}(\alpha/2)\right)/(4\alpha).
$$
We list  in Table \ref{Sec6:Tab1Eg3}  the  relative errors  and the order of convergence of the quadrature formula $\mathcal{Q}^{\alpha}_{n,4}[f]$, and in Table \ref{Sec6:Tab2Eg3} those of the quadrature formula $\mathcal{Q}^{\alpha}_{n}[f]$.
According to the results reported in  Tables \ref{Sec6:Tab1Eg3} and  \ref{Sec6:Tab2Eg3}, we confirm that the quadrature formulas proposed in this paper are efficiency for the integrals defined as in \eqref{Sec2:Int}, which is a piecewise continuous function, even though the standard deviation is very small.

}\end{example}

\vspace{-1em}
\begin{center}
\footnotesize
\makeatletter
\def\@captype{table}
\makeatother
{\caption{Relative errors and convergence orders of QuadP for the function \eqref{Sec6:fun}}
\vspace{0.2cm}
\label{Sec6:Tab1Eg3}}{
\begin{tabular}{c|c|c|c|c|c|c|c}
\hline
\multirow{2}{*}{$ \alpha$}&\multirow{2}{*}{$c_0$}
&\multicolumn{2}{c|}{ $n=4$ }&\multicolumn{2}{c|}{$n=12$}
&\multicolumn{2}{c}{ $n=16$ }\\
\cline{3-8}
&&RE&Order&RE&Order&RE&Order\\
\hline
100         &7.07e-3 &1.96e-16&8.88&1.96e-16&8.88&1.96e-16&8.88\\
1000       &7.07e-4 &4.89e-16&6.12&2.45e-16&6.22&3.67e-16&6.16\\
10000     &7.07e-5 &4.59e-16&4.85&6.12e-16&4.82&1.53e-16&4.97\\
100000   &7.07e-6 &9.56e-16&4.01&1.91e-16&4.15&3.82e-16&4.09\\
1000000 &7.07e-7 &2.15e-15&3.45&2.39e-16&3.61&1.19e-16&3.66\\
\hline
\end{tabular}}
\end{center}

\vspace{-1em}
\begin{center}
\footnotesize
\makeatletter
\def\@captype{table}
\makeatother
{\caption{Relative errors and convergence orders of QuadE for the function \eqref{Sec6:fun}}
\vspace{0.2cm}
\label{Sec6:Tab2Eg3}}{
\begin{tabular}{c|c|c|c|c|c|c|c}
\hline
\multirow{2}{*}{$ \alpha$}&\multirow{2}{*}{$ c_0$}
&\multicolumn{2}{c|}{ $n=3$ }&\multicolumn{2}{c|}{$n=4$}
&\multicolumn{2}{c}{ $n=5$ }\\
\cline{3-8}
&&RE&Order&RE&Order&RE&Order\\
\hline
2000         & 3.54e-4  &3.67e-16&5.22&4.89e-16&5.18&2.45e-16&5.27\\
20000       &  3.54e-5 &2.45e-15&4.12&1.22e-15&4.19&3.06e-16&4.32\\
200000     & 3.54e-6  &1.53e-14&3.42&1.91e-16&3.76&3.82e-16&3.71\\
2000000   &  3.54e-7&2.44e-14&3.02&3.58e-15&3.15&9.56e-16&3.24\\
20000000 & 3.54e-8 &9.05e-14&2.68&3.29e-15&2.87&4.48e-16&2.99\\
\hline
\end{tabular}}
\end{center}



}

\end{document}